\documentclass[12pt,a4paper]{article}
\usepackage[english]{babel} 
\usepackage[T1]{fontenc}
\usepackage[pdftex]{hyperref}
\usepackage[all]{hypcap}
\usepackage[all]{xy}
\usepackage{dsfont,url,graphicx,framed,amssymb,amsfonts,amsmath,amsthm,amsbsy,amssymb,textcomp,pdfpages,fancyhdr,csquotes,mathtools,bbm,a4wide,babelbib,eufrak,cite,caption,subcaption,tabu}
\usepackage[subnum]{cases}

\newtheorem{lemma}{Lemma}[section]
\newtheorem{proposition}[lemma]{Proposition}
\newtheorem{theorem}[lemma]{Theorem}
\newtheorem{conjecture}[lemma]{Conjecture}
\newtheorem{corollary}[lemma]{Corollary}
\newtheorem{assumption}[lemma]{Assumption}
\newtheorem{remark}[lemma]{Remark}
 \newtheoremstyle{TheoremNum}
        {\topsep}{\topsep}              
        {\itshape}                      
        {}                              
        {\bfseries}                     
        {.}                             
        { }                             
        {\thmname{#1}\thmnote{ \bfseries #3}}
    \theoremstyle{TheoremNum}
\newtheorem{thmn}{Theorem}
\newtheoremstyle{PropNum}
        {\topsep}{\topsep}              
        {\itshape}                      
        {}                              
        {\bfseries}                     
        {.}                             
        { }                             
        {\thmname{#1}\thmnote{ \bfseries #3}}
    \theoremstyle{PropNum}
\newtheorem{propn}{Proposition}

\def\N{\mathds{N}}
\def\Z{\mathds{Z}}
\def\Q{\mathds{Q}}
\def\R{\mathds{R}}

\def\F{\mathds{F}}

\def\m{\mathfrak{m}}

\def\O{\mathcal{O}}
\def\A{\mathcal{A}}
\def\G{\mathcal{G}}

\def\tr{\textnormal{tr}\,}
\def\v{\upsilon}

\newcommand\ceil[1]{\lceil #1 \rceil}
\newcommand\floor[1]{\lfloor #1 \rfloor}
\newcommand\frob[1]{\textnormal{Frob}_{#1}}
\newcommand\limd[1]{\ensuremath{\lim_{#1 \rightarrow \mid \infty}}}

\newcommand\ctext[1]{#1}

\renewcommand\d[1]{\ensuremath {\,\mathrm{d}#1}}
\renewcommand\min[2]{\textnormal{min}\{#1,#2\}}

\newcommand{\Addresses}{{
  \bigskip
  \footnotesize
  \textsc{{\bf KU Leuven} \\ Department of Mathemetics \\ Celestijnenlaan 200 B \\ B-3001 Heverlee\\ BELGIUM  } \vspace{8 pt} \\
  \textsc{{\bf Universit\'e du Luxembourg} \\ Mathematics Research Unit FSTC \\ 6, rue Richard Coudenhove-Kalergi \\ L-1359 Luxembourg \\ LUXEMBOURG  } \vspace{8 pt} \\
  \texttt{jasper.vanhirtum@wis.kuleuven.be}
}}

\author{Jasper Van Hirtum}
\title{On the Distribution of Frobenius of Weight $2$  Eigenforms with   Quadratic Coefficient Field}

\begin{document}
\maketitle
\begin{abstract}
	In this article we present a heuristic model that describes the asymptotic behaviour of the number of primes $p$ such that the $p$-th coefficient of a given  eigenform is a rational integer. We treat the case of a weight $2$ eigenform with quadratic coefficient field without inner twists. Moreover we present numerical data which agrees with our model and the assumptions we made to obtain it.
\end{abstract}
\section{Introduction}

Let $f$ be a weight $2$ cuspidal Hecke eigenform of level $\Gamma_1(N)$ with quadratic coefficient field and without inner twist.  Denote the $p$-th coefficient of the standard $q$-expansion of $f$ by $a_p(f)$. Then the set of primes $\{p\  |\, a_p(f) \in \Q\}$ is known to be of  density zero, cf.~\cite[Corollary~1.1]{KSW08}. Part of the conjecture that Kumar  Murty posed based on earlier work of S. Lang and H. Trotter  \cite{LaTr} is the following.
\begin{conjecture}[Conjecture 3.4 \cite{Mur99}]\label{conj:Murty}
	Let  $f$ be a weight $2$ normalized cuspidal Hecke eigenform of level $\Gamma_1(N)$ with quadratic coefficient field and without inner twists. Then
	$$\#\{ p<x \textnormal{ prime} \mid a_p(f) \in \Q \} \sim c_f\frac{\sqrt x}{\log x},$$
	with $c_f$ a constant that depends on the eigenform. 
\end{conjecture}
In this paper we present a heuristic model that makes this conjecture explicit. More precisely we will prove the following theorem.
\begin{thmn}[\ref*{theorem:MainResult}]
 \ctext{Let  $f$ be a weight $2$ normalized cuspidal Hecke eigenform of level $\Gamma_1(N)$ with quadratic coefficient field $\Q(\sqrt D)$ and without inner twists. Assume  that there exists a positive integer $m_0$ such that  Assumptions  \ref{ass:PmAss} and \ref{ass:ind}  hold for $f$ and all positive integers in $m_0\Z$}. Then there is an explicit constant $\widehat F$, depending on the images of the Galois representations attached to $f$, such that Conjecture \ref{conj:Murty} holds with 
$$c_f = \frac{16 \sqrt D \widehat F} {3\pi ^{2}}.$$
\end{thmn}
Our work is based on the methods used by Serge Lang and Hale Trotter in  \cite{LaTr} where they derive a heuristic model for the behaviour of the of the coefficients of the $L$-polynomial of an elliptic curve, i.e., the coefficients of a weight $2$ eigenform with rational coefficients. 

Section \ref{sect:Prel} contains some facts concerning modular forms. In Section \ref{sect:Model} we describe the assumptions  needed to reduce our problem to the product of two functions: one concerning the real absolute value  and the other derived from the non-archimedean places. In   Section \ref{section:suth} we will discuss the factor of the infinite place. \ctext{For this factor we will use recent results on the Sato-Tate conjecture for abelian surfaces and one additional assumption}. The factor at the finite places will be discussed in Section \ref{section:FinPla}. We will derive this factor from the adelic representation attached to the eigenform. Section \ref{section:MainResult} contains the proof of our main result. In the final section we compare our model to numerical data. Moreover we check all assumptions and intermediate results numerically. All computations agree with our model and the assumptions we made to obtain these results. Therefore we are led to believe our heuristic model correctly predicts \ctext{the asymptotic number} of primes with rational integer coefficient.

\begin{remark}
Let $f$ be a  cuspidal Hecke eigenform of level $\Gamma_1(N)$.
\begin{enumerate}
\item If $f$ has CM by the  Dirichlet character $\chi$, then  $\chi(p) a_p (f) = a_p(f)$ for almost all primes. If $p$ is a prime such that $\chi(p) \neq 1,0$ then $a_p(f) = 0$.  By Dirichlet's theorem of arithmetic progression the density of the set $\{p \textnormal{ prime } | a_p(f) =0\}$ is at least $\frac{1}{2}$. 
\item Suppose that  $f$ has quadratic coefficient field $K_f$ and an inner twist by the Dirichlet character $\chi$  and non-trivial automorphism $\sigma \in Gal(K_f/\Q)$. Let $p$ be a prime such that $\chi(p) = 1$ then 
$\sigma (a_p(f)) = \chi(p) \cdot a_p(f) = a_p(f)$ so $a_p(f) \in \Q$. If $n$ is the modulus of the character $\chi$ and $p$ is a prime such that $p \equiv 1 \mod n$, then $\chi(p) = 1$ and $a_p(f) \in \Q$. Again by Dirichlet's theorem of arithmetic progression the density of the set $\{p \textnormal{ prime } | a_p(f) \in \Q\}$ is at least $\frac{1}{\phi(n)}$.
\item Let  $f$ be a weight $k$ form without inner twists and let $d$ be the extension \ctext{degree of  $K_f$} over $\Q$. Then Kumar Murty conjectured the following \cite[Conjecture~3.4]{Mur99}
$$  \#\{p<x\textnormal{ prime } \mid a_p(f) \in \Q \} \sim c_f \begin{cases} \sqrt x / \log x & \textnormal{if } k = d= 2, \\ \log \log x & \textnormal{if } k=2 \textnormal{ and } d = 3, \\ &\textnormal{or } k = 3 \textnormal{ and } d = 2,\\ 1 &\textnormal{else}. \end{cases}$$
\end{enumerate}
\end{remark}

\subsection*{Acknowledgements}
I wish to express my sincere gratitude to Jan Tuitman and Gabor Wiese for suggesting me this problem, for our discussions, for  their enthusiasm  and for their  guidance.

\ctext{I would also like to thank Andrew Sutherland and the anonymous referee for  useful comments}.

\section{Preliminaries}\label{sect:Prel}
In this section we some  recall basic facts concerning modular forms.
\begin{lemma}\label{lemma:prel0}
Let $f$ be a weight $2$ eigenform of level $\Gamma_1(N)$ and trivial nebentypus. If $N$ is square-free, then $f$ does not have inner twists.
\end{lemma}
\begin{proof}
This follows from \cite[Theorem~3.9~bis]{Ri80}.
\end{proof}
\begin{lemma}\label{lemma:prel}
	Let $f$ be a  normalized cuspidal Hecke eigenform of level $\Gamma_1(N)$. 
	\begin{enumerate}
		\item If $f$ has trivial nebentypus, then  the coefficient field of $f$ is totally real.
		\item If $f$ does not have any inner twist, then $f$ has trivial nebentypus.
	\end{enumerate}
\end{lemma}
\begin{proof}
	Let $K_f$ be the coefficient field of $f$, $N$ the level  and $\varepsilon$ the nebentypus of $f$. Let $<\cdot, \cdot>$ be the Petersson scalar product. The Hecke operators are self adjoint with respect to $<\cdot,\cdot>$ up to the character $\varepsilon$ \cite[Theorem 5.1]{Lang76}, i.e.,
	$$<T_p\,\cdot,\cdot> = \varepsilon(p) <\cdot,T_p\,\cdot>,$$
	for all primes $p$ not dividing $N$.
	Hence for any such prime $p$ we obtain
	\begin{align*}
		a_p(f)<f,f> &= <T_pf,f>\\
		&=\varepsilon(p)<f,T_pf> \\
		&= \varepsilon(p)\widehat{a_p(f)}<f,f>,
	\intertext{where $\widehat {a_p(f)}$  denotes the complex conjugate of $a_p(f)$.
	In particular }
		\varepsilon^{-1}(p) a_p(f)& = \widehat{a_p(f)},
	\end{align*}
	for all primes $p$ not dividing $N$.
	\begin{enumerate}
		\item If the nebentypus of $f$ is trivial, then by the above
		$$a_p(f) = \widehat{a_p(f)}$$
		for all primes $p$ not dividing $N$. In particular \ctext{$K_f\subset \R$}.
		\item Note that  $\widehat{\,\cdot\,}|_{K_f} \in \textnormal{Aut}_\Q(K_f)$ so if  $\varepsilon$ is not the  trivial character, then $f$ has  inner twist by $\varepsilon^{-1}$.
	\end{enumerate}
\end{proof}

\section{Heuristic model}\label{sect:Model}
For the remainder of this article $f$ will be a   normalized   cuspidal  Hecke eigenform of weight $2$ and level $\Gamma_1(N)$  without inner twist or CM and with quadratic coefficient field $K_f$. By  Lemma  \ref{lemma:prel} $K_f\subset \R$ and $f$ has trivial nebentypus. Let  $D$ be the positive square-free integer such that $K_f = \Q(\sqrt D)$. Denote by $\overline{\, \cdot \,}$ the unique non-trivial element of the Galois group of $K_f/\Q$. Define
$$Z_p := a_p(f) -\overline{a_p(f)}.$$
Note that $Z_p \in \sqrt D \Z$ since $a_p(f)$ is an algebraic integer in $K_f$. Moreover 
\begin{align*}
	a_p(f) \in \Q &\Leftrightarrow Z_p = 0 \\
	& \Leftrightarrow  \frac{-m\sqrt D} 2<Z_p < \frac{m \sqrt D} 2 \textnormal{ and } Z_p \equiv 0 \mod m\sqrt D \Z \ \textnormal{ for all } m \in \N .
\end{align*}
\ctext{In other words  the condition $a_p(f) \in \Q$ is equivalent to a condition on the real and $\ell$-adic absolute value of $Z_p$ for finite places $\ell$ dividing  $m$ for any positive integer $m$}.  Denote  $\pi(x) := \#\{p <x  \textnormal{ prime}\}$ and 
\begin{align*}
	P(x) &:= \frac{\#\{ p  <x \textnormal{ prime} \mid Z_p =0\}}{\pi(x)},\\
	P_m(x) & := \frac{\#\{ p <x \textnormal{ prime} \mid Z_p \equiv 0 \mod m\sqrt D \Z\}}{\pi(x)},\\
	P^m(x) & := \frac{\#\big\{ p <x \textnormal{ prime} \mid Z_p \in ]\frac{-m\sqrt D} 2,\frac{m\sqrt D} 2[\big\}}{\pi(x)}.
\end{align*}
Since $|a_p(f)| = \O(\sqrt p)$ cf. \cite[Lemma 2]{Lang76} 
$$\lim_{m\rightarrow \infty} P_m(x) = P(x) \textnormal{ and } \lim_{m\rightarrow \infty}P^m(x) = 1 $$
for all $x>2$. In particular
\begin{align*} &\lim_{m\rightarrow  \infty }  \frac{P^m(x) \cdot P_m(x)} {P(x)} = 1  \ \textnormal{ for all } x>2 \intertext{ so}
 &\lim_{x\rightarrow \infty}  \lim_{m\rightarrow  \infty }  \frac{P^m(x) \cdot P_m(x)} {P(x)} = 1.
\end{align*}
Our first assumption states that the order of the double  limit can be reversed.

\begin{assumption}\label{ass:ind}
	Let $f$ be as above. Then
	$$ \limd m \lim_{x\rightarrow \infty} \frac{P^{m}(x)\cdot P_{m}(x)}{P(x)} = 1,$$
	where $ \limd m $ denotes that the limit over $m$ is taken by divisibility.
\end{assumption}
We say that $a$ is the limit of a  series $\{a_m\}_\N$ by divisibility if for all $\varepsilon>0$ there exists a positive integer $m_0$ such that for all $m\in m_0 \N$
$$|a_m-a|<\varepsilon.$$

In Section \ref{section:MainResult} we will show that the convergence of the double limit in Assumption \ref{ass:ind} follows from the  weaker condition that there exists at least one positive integer $m$ satisfying the following assumption. 
\begin{assumption}\label{ass:ind1}
Let $f$ be as above and  $m$  a positive integer. Then   
$$\lim_{x\rightarrow \infty} \frac{P^{m}(x)\cdot P_{m}(x)}{P(x)}= \alpha_m, $$
with $0<\alpha_m<\infty$.
\end{assumption}
Note that the $0<\alpha_m$ part of the statement will follow immediately from Lemma \ref{lemma:chebdens}. Assumption \ref{ass:ind1} is enough to prove the asymptotic behaviour of $\#\{p < x \text{ prime} \mid a_p \in \Q\}$. However to make the constant $c_f$ explicit we will need the stronger Assumption \ref{ass:ind}.

By deriving suitable expressions for the arithmetic part $P_m(x)$ and the real part $P^m(x)$ respectively we will obtain the asymptotic behaviour of $P(x)$ predicted by Conjecture \ref{conj:Murty} from Assumption \ref{ass:ind1}. Additionally under the stronger condition of Assumption \ref{ass:ind} we will obtain an explicit constant. In Section \ref{section:FinPla} we use Chebotarev's density theorem  to prove an explicit formula for the factor $P_m(x)$. For the factor at the infinite place we will need  additional assumptions. We describe the assumptions and the results that follow in the next section.

\section{The place at infinity}\label{section:suth}
In this section we describe a heuristic formula  for the factor 
$$P^m(x) = \frac{\#\big\{ p <x \textnormal{ prime} \mid Z_p \in ]\frac{-m\sqrt D} 2,\frac{m\sqrt D} 2[\big\}}{\pi(x)},$$
\ctext{which we derive from natural assumptions}.  We use results from a recent paper by F. Fit\'e, K. Kedlaya, V. Rotger and A. Sutherland that describes the joint distribution of the coefficients of the normalized $L_p$-polynomial of hyperelliptic curves of genus $2$ under the  assumption of  the Sato-Tate conjecture for abelian varieties (cf.  \cite{FKRS}). Note that we  use  the Sato-Tate conjecture for abelian varieties rather than the proven   Sato-Tate distribution for modular forms. The latter describes the distribution of the real absolute value of the coefficients but claims nothing about the coefficients as  elements of the number field $K_f$. 

Let $f$ be as above. Then one can associate via  Shimura's construction (cf.~\cite[section 1.7]{DiSh}) an abelian variety $\A_f$ of dimension $[K_f:\Q]$ to the eigenform $f$. For every prime $p$ the $L_p$-polynomial associated to the variety $\A_f$ splits as the $L_p$-polynomial of the eigenform and its Galois conjugate over $K_f$. More precisely, let $L_p(T):=p^2T^4 +pX_pT^3 +Y_pT^2+X_pT+1 $ be the $L_p$-polynomial of $\A_f$ then
$$L_p(T) = (pT^2 - a_p(f)T+1)(pT^2-\overline{a_p(f)} T+1).$$
Hence 
$$a_p(f) = -\frac{X_p}{2} \pm \sqrt{2p-Y_p+\frac{X_p^2} 4}.$$
Note that from the $L_p$-polynomial of $\A_f$ we cannot deduce $a_p(f)$ completely. Indeed we  only obtain its  Galois orbit. However we can decide whether or not $Z_p$ lies in a symmetrical interval around zero since $$|Z_p| = \sqrt{8p-4Y_p+X^2_p}.$$

Let $a_{1,p}=\frac{X_p}{\sqrt p}$ and $a_{2,p} = \frac{Y_p}{p}$ be the coefficients of the normalized $L_p$-polynomial $L_p( T/\sqrt p)$. Then 
$$Z_p \in \Big]-\frac{m\sqrt D} 2,\frac{m\sqrt D} 2 \Big[  \Leftrightarrow \sqrt{2-a_{2,p}+a_{1,p}^2/4}< \frac{m\sqrt D }{4\sqrt p} .$$

The generalized Sato-Tate conjecture states that this distribution  is completely determined by the so called Sato-Tate group. In \cite{FKRS} Fit\'e et al. study the joint distribution of $(a_{1,p},a_{2,p})$ for abelian surfaces.   More precisely they prove the following theorem. 
\ctext{\begin{theorem}\label{theorem:FKRS}
Let $\A$ be an abelian surface. There exist exactly $52$  Sato-Tate groups for abelian surfaces, of which only $34$ occur over $\Q$. Moreover the conjugacy class of the Sato-Tate group of $\A$ is uniquely determined by its
 Galois type (cf.~\cite[Def.~1.3~]{FKRS}) and vice versa.
\end{theorem}}
\begin{proof}
 This is Theorem 1.4 in  \cite{FKRS}.
\end{proof}
\begin{corollary} \label{cor:PmInf}
Let $f$ and $\A_f$ be as above. \ctext{The generalized Sato-Tate conjecture holds for $A_f$ and the joint distribution} of $(a_{1,p}, a_{2,p})$ is given by $$\Phi:T\rightarrow \R :(x,y) \mapsto \frac{1}{2\pi ^2}\sqrt{\frac{(y-2x+2)(y+2x+2)} {x^{2}-4y+8}},$$
where $T$ is the subset of the plane (Fig. \ref{fig:IntRange})
$$T:=\left\{ (x,y)|\, y+2>|2x| \textnormal{ and } 4y< x^{2}+8  \right \}.$$ 	
Moreover denote by $\delta$ the measure with density $\Phi$ and let $S$ be a measurable set. Then $$\delta(S) \sim \frac{\#\{p <x \textnormal{ prime} \mid (a_{1,p},a_{2,p}) \in S\}}{\pi(x)}.$$
\end{corollary}
\begin{proof}
The $\Q$-algebra of endomorphisms  of $\A_f$ over $\Q$ is $K_f$. Moreover the $\Q$-algebra of endomorphisms of $\A_f$ over $\overline \Q$ is also $K_f$ since  $f$ does not have any inner twists (cf.~\cite[Theorem 5]{RibetEndo}). So all endomorphisms of $\A_f$ over $\overline \Q$ are already defined over $\Q$. In particular the Galois type of $\A_f$ is 
$$[\textnormal{Gal}(\Q/\Q), K_f \otimes_\Z \R ] = [1,\R\times \R],$$
since $K_f$ is a real quadratic number field. \ctext{The generalized Sato-Tate conjecture is proven for abelian surfaces of this Galois type by Christian Johansson in \cite{Joha}[Proposition 22].} It follows from Tables $8$ and $11$ in \cite{FKRS}  that the Sato-Tate group of $\A_f$ is $SU(2)\times SU(2)$. Finally, the joint distribution function of this group is given by \cite[Table~5]{FKRS}.
\end{proof}
The following result is proven  by K. Koo, W. Stein and G. Wiese (cf.~\cite[Corollary~1.1]{KSW08}). \ctext{We give an alternative proof using recent results on Sato-Tate equidistribution}.
\ctext{\begin{corollary}
Let $f$ be as above. The set
$\{p <x  \textnormal{ prime} \mid a_p(f)\in \Q\}$ has density zero.
\end{corollary}}
\begin{proof}
Let $S = \{(x,y) \mid \sqrt{x^2-4y+8}=0\}$ then by Corollary \ref{cor:PmInf} 
$$\frac{\#\{p <x \textnormal{ prime}\mid a_p(f)\in \Q\} }{\pi(x)}= \frac{\#\{p <x \textnormal{ prime} \mid (a_{1,p},a_{2,p})\in S\}}{\pi (x) } \sim \delta(S).$$
Clearly, $\delta(S)=0$.
\end{proof}

\begin{figure}[h]
\centering
\includegraphics[width= 0.5\textwidth]{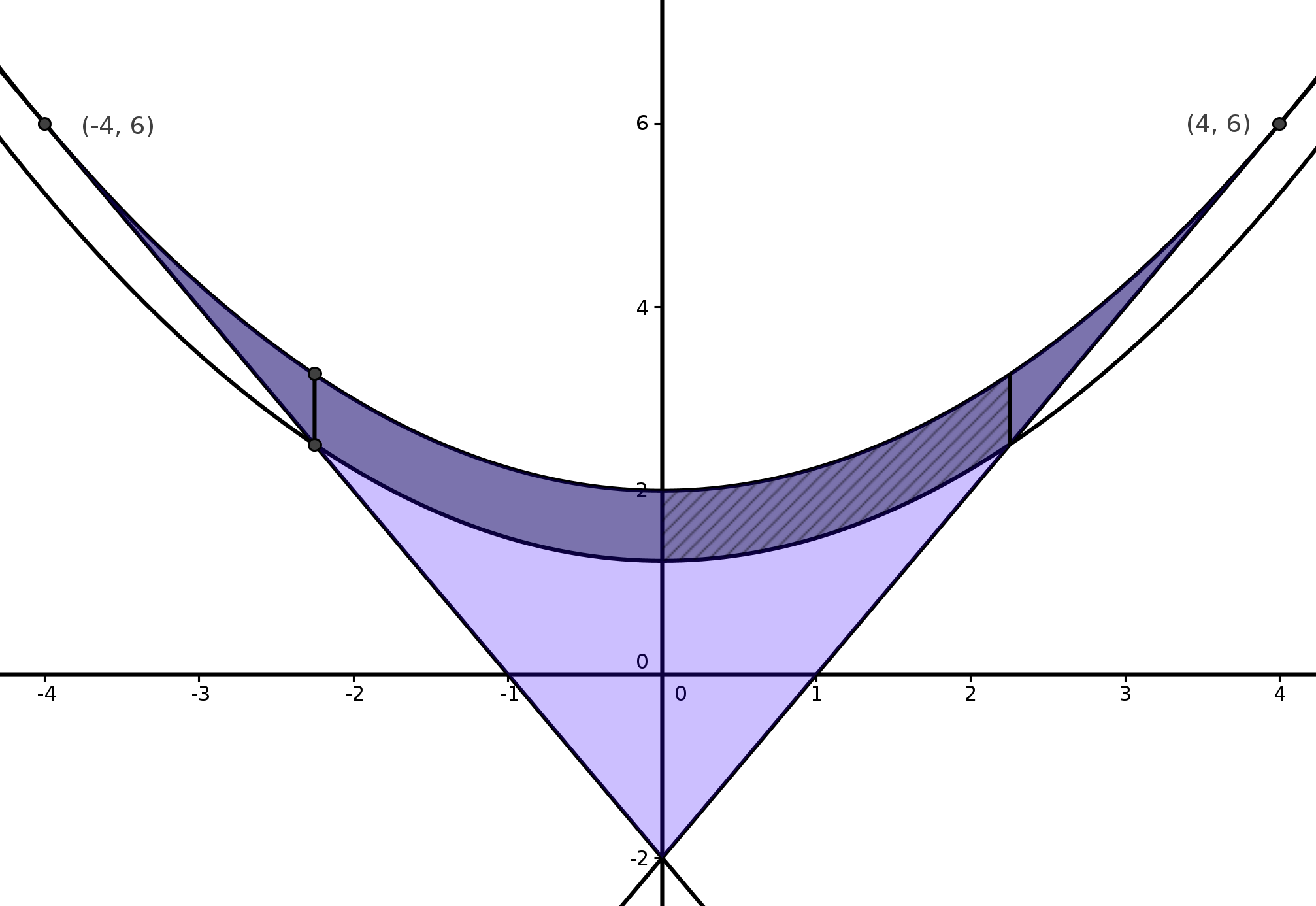}
\caption{The areas $T$ (dark and light blue),  $T_\varepsilon$ (dark blue) and $T_\varepsilon^1$ (hatched).} 
\label{fig:IntRange}
\end{figure}

Define for any $\varepsilon>0$
$$T_{\varepsilon} := \left\{(x,y) \in T \mid \sqrt{x^2/4-y + 2} < \varepsilon \right\}.$$
Then 
$$Z_p \in \Big]-\frac{m\sqrt D} 2,\frac{m\sqrt D} 2 \Big[  \Leftrightarrow  (a_{1,p},a_{2,p}) \in T_{m\sqrt D/ (4\sqrt p) }.$$

\ctext{By Corollary \ref{cor:PmInf}}
$$\delta(T_{{m\sqrt D}/(4\sqrt p)}) \sim \frac{ \#\Big\{q < x  \textnormal{ prime} \mid \sqrt{2 - a_{2,q} + a_{1,q}^2/4} < \frac{m\sqrt D}{4\sqrt p}\Big\}}{\pi(x)}$$
since $T_{m\sqrt D/ (4\sqrt p)}$ is a measurable set. \ctext{We will need the following heuristic assumption}.
\begin{assumption}\label{ass:PmAss}
Let $m$ be a positive integer. Then 
$$P^m(x) \sim \frac{1}{\pi(x)}\sum^x_{p=2} {\delta(T_{m \sqrt D/(4\sqrt p)})}$$
where the sum is taken only over primes. 
\end{assumption} 
The idea behind this assumption is that  approximating the probability of $\left|\frac{Z_p}{2\sqrt p}\right| < \frac{m\sqrt D}{4\sqrt p}$ (which is either $1$ or $0$) by the probability that $\left|\frac{Z_q}{2\sqrt q}\right| < \frac{m\sqrt D}{4\sqrt p}$ for any prime $q$ is 'good on average'. Note that for each individual prime $p$ this approximation is bad. However the assumption states that summing over all primes $p$ does yield a good approximation.  

\begin{lemma}\label{lemma:DensT}
The measure of the set $T_{\varepsilon}$ is 
$$\delta (T_{\varepsilon}) = \frac{32}{3\pi^2}\varepsilon +o \left(\varepsilon \right).$$
\end{lemma}
\begin{proof}
It suffices to bound the integral
$$\delta (T_\varepsilon ) = \int_{T_\varepsilon} \Phi \d A.$$
 Since the density function $\Phi$ is even with respect to the first variable $x$ we can restrict to positive $x$. We split the integration domain $T_\varepsilon$ in two parts (Fig. \ref{fig:IntRange})
\begin{align*}
T_\varepsilon^1 &:= \{(x,y) \in T_\varepsilon \mid 0<x<4-2\varepsilon\} \textnormal{ and}\\
T_\varepsilon^2 &:= \{(x,y) \in T_\varepsilon \mid 4-2\varepsilon<x<4\}.
\end{align*} 
So 
$$\delta  (T_\varepsilon ) =2\int_{T^1_\varepsilon} \Phi \d A +  2\int_{T^2_\varepsilon} \Phi \d A.$$

Parametrizations of both sets are given respectively by
\begin{align*}
T_\varepsilon^1 & = \{(u,u^2/4+2-v^2) \mid 0<u<4-2\varepsilon\ \wedge \ 0<v<\varepsilon\} \textnormal{ and} \\
T_\varepsilon^2 & = \{(u,u^2/4+2-v^2) \mid 4-2\varepsilon<u<4\ \wedge \ 0<v<2-u/2\}.
\end{align*}
The determinant of the Jacobian of the parametrization, $\d A$,  is $2v \d v\d u$ so 
\begin{align*}
&2\Phi(u,u^2/4+2-v^2) \d A \\& =  \frac{2}{2\pi ^2} \sqrt{\frac{(u^2/4+2-v^2 - 2u+2)(u^2/4+2-v^2 + 2u+2)}{u^2-u^2-8+4v^2 +8} }2v \d v \d u \\
  &=  \frac{1}{4\pi ^2}\sqrt{(u^2+16-4v^2-8u)(u^2+16-4v^2+8u) } \d v \d u \\
  &=   \frac{1}{4\pi ^2}\sqrt{(u^2+16-4v^2)^2-64u^2 } \d v \d u.
\end{align*}  
First we show that $$ 2\int_{T^2_\varepsilon} \Phi \d A = \O(\varepsilon^2).$$
It suffices to show that the limit of $\frac{1}{\varepsilon^2} 2\int_{T^2_\varepsilon} \Phi \d A$ is finite if $\varepsilon$ tends to zero. Let us compute
\begin{align*}
	\lim_{\varepsilon \rightarrow 0} \frac{1}{\varepsilon^2} 2\int_{T^2_\varepsilon} \Phi \d A & =\lim_{\varepsilon \rightarrow 0} \frac{1}{\varepsilon^2} \frac{1}{4\pi^2} \int_{4-2\varepsilon}^4 \int_0^{2-u/2} \sqrt{(u^2+16-4v^2)^2-64u^2 } \d v \d u.
	\intertext{Then by l'H\^{o}pital's rule we obtain }
	\lim_{\varepsilon \rightarrow 0} \frac{1}{\varepsilon^2} 2\int_{T^2_\varepsilon} \Phi \d A & \stackrel{\widehat{\textnormal{H}}}{=} \lim_{\varepsilon \rightarrow 0} \frac{1}{\varepsilon} \frac{1}{2\pi^2}  \int _ 0^{\varepsilon}   \sqrt{(\varepsilon^2+16-4v^2)^2-64\varepsilon^2 } \d v \\
	& < \lim_{\varepsilon \rightarrow 0} \frac{1}{\varepsilon} \frac{1}{2\pi^2}  \int _ 0^{\varepsilon}   \sqrt{(\varepsilon^2+16)^2-64\varepsilon^2 } \d v \\
	& = \lim_{\varepsilon \rightarrow 0}\frac{1}{2\pi^2}\sqrt{(\varepsilon^2+16)^2-64\varepsilon^2 }  \\
	& = \frac{8}{\pi^2}< \infty.
\end{align*} 
In particular we obtain 
$$\delta(T_\varepsilon) = 2 \int_{T_\varepsilon^1} \Phi \d A + \O(\varepsilon^2).$$
Finally, denote $\phi(u,v) = 2\Phi \frac{ \d A}{\d v \d u} =\frac{1}{4\pi^2} \sqrt{(u^2+16-4v^2)^2-64u^2}$. Then, again by l'H\^{o}pital's rule and Leibniz rule for double integration
\begin{align*}
\lim_{\varepsilon \rightarrow 0} \frac{\frac{32}{3\pi^2} \varepsilon - \delta(T_\epsilon) } {\varepsilon} &= \lim_{\varepsilon \rightarrow 0} \frac 1 \varepsilon \left( \frac{32}{3\pi^2} \varepsilon - \int_{0}^{4-2\varepsilon} \int _0^\varepsilon \phi(u,v) \d v \d u \right) \\
&  \stackrel{\widehat{\textnormal{H}}}{=} \lim_{\varepsilon \rightarrow 0} \frac{32}{3\pi^2}  -\int_{0}^{4-2\varepsilon} \phi(u,\varepsilon) \d u + 2\int _0^\varepsilon \phi(4-2\varepsilon,v) \d v \\
& = \frac{32}{3\pi^2}  -\int_{0}^{4} \phi(u,0) \d u + 0 \\
& =  \frac{32}{3\pi^2}  - \frac{1}{4\pi ^2} \int_{0}^{4} \sqrt{(u^2+16)^2 -64u ^2} \d u \\
&  =  \frac{32}{3\pi^2}  - \frac{1}{4\pi ^2} \int_{0}^{4} (16- u ^2) \d u \\
&=  \frac{32}{3\pi^2}  - \frac{1}{4\pi ^2} \Big[16u-\frac{u^3}{3}\Big]_0^4 \\
&=  0. \qedhere
\end{align*}
\end{proof}

\ctext{\begin{corollary} \label{cor:PmCor}
For all $m>0$ satisfying Assumption \ref{ass:PmAss} 
$$P^m(x) \sim \frac{16m\sqrt D}{3\pi^2 \sqrt x} .$$
\end{corollary}}
\begin{proof}
By Assumption \ref{ass:PmAss}, Lemma \ref{lemma:DensT}  with $\varepsilon = m\sqrt D/(4\sqrt p)$  and the fact that $\sum_{p=2}^x \frac 1 {2\sqrt p} \sim \frac{\sqrt x }{\log x}$ respectively we obtain
\begin{align*}
P^m(x) & \sim \frac 1 {\pi(x)} \sum_{p=2}^x \delta(T_{m\sqrt D /(4\sqrt p)}) \\
 & \sim \frac1{\pi(x)} \frac{16m\sqrt D }{3\pi^2} \sum_{p=2}^x \frac 1 {2\sqrt p } \\
 & \sim \frac{16m\sqrt D }{3\pi^2} \frac{\sqrt x}{\pi(x) \log x} \\
 & \sim \frac{16m\sqrt D}{3\pi^2\sqrt x} . \qedhere
\end{align*}
\end{proof}

\section{Finite places}\label{section:FinPla}
In this section we describe the remaining factor $P_m(x)$. No heuristic is needed to obtain the results of this section. For each positive integer $m$ the factor  $P_m$ can be computed by Chebotarev's density theorem and  the image of the  mod $m$ Galois representation attached to $f$.

Let $f$ be as above and denote the ring of integers of its coefficient field by $\O_f$. Let  $m$ be a positive integer and denote $\O_m := \O_f \otimes _\Z \Z/m\Z$. Denote the absolute Galois group of $\Q$ by $G_\Q$. Then the action of $G_\Q$ on the $m$-torsion points of $\A_f$ induces a mod-$m$ representation 
$$\rho_m :G_\Q \rightarrow GL_2(\O_m).$$
By taking the inverse limit over all integers $m$ we obtain an adelic representation
$$\widehat \rho :=\varprojlim_{m} \rho_m: G_\Q \rightarrow GL_2(\widehat \O),$$
where  $\widehat \O = \O_f \otimes _\Z \widehat \Z$ is the ring of finite adeles of $K_f$.
If $\ell$ is a prime we obtain an $\ell$-adic representation  by taking the  limit over all powers of $\ell$ torsion points
$$\widehat \rho_\ell := \varprojlim_k \rho_{\ell^k} : G_\Q \rightarrow GL_2(\widehat \O_\ell),$$
with $\widehat \O_\ell = \O_f\otimes_\Z \Z_\ell$. Note that  by definition $\widehat \rho = \prod_\ell \widehat \rho_\ell$. Moreover for any positive integer $m$ and any prime  $\ell$ dividing $m$   the following diagram commutes

$$ \xymatrix{  & G_\Q \ar[dd]^{\widehat \rho } \ar[dddl]_{\widehat \rho_\ell}  \ar[dddr]^{\rho_m}& \\  \\
				   & GL_2(\widehat \O)  \ar[dl] \ar[dr] \ar[d]& \\
				   GL_2(\widehat\O_\ell) \ar[r]& GL_2(\O_\ell)  &\ar[l] GL_2(\O_m).		
			}  
	$$
Let $\overline{\, \cdot \, } $ be the unique non-trivial element of the Galois group of $K_f$ over $\Q$. Then $\overline{\, \cdot \, } $ induces by the tensor product  endomorphisms on $\widehat \O$, $\widehat \O_\ell$ and $\O_m$. By abuse of notation we denote each of these morphisms by $\overline{\, \cdot \, } $. Hence we obtain the following maps
\begin{align*}
Z&:GL_2(\widehat\O) \rightarrow \widehat \O: \sigma \mapsto \tr \sigma - \overline{\tr \sigma }, \\
Z&:GL_2(\widehat \O_\ell) \rightarrow \widehat \O_\ell: \sigma \mapsto \tr \sigma - \overline{\tr \sigma }\ \textnormal{ and}\\
Z&:GL_2(\O_m) \rightarrow \O_m: \sigma \mapsto \tr \sigma - \overline{\tr \sigma }.
\end{align*}

Consider the following (subsets) of the images of the representations
\begin{align*}
\widehat \G &:= \textnormal{Im}\ \widehat\rho, & \widehat \G^t &:=\{\sigma \in \widehat \G \mid Z(\sigma) = 0\},\\
\widehat \G_\ell &:= \textnormal{Im} \ \widehat \rho_\ell, & \widehat \G_\ell^t &:= \{\sigma \in \widehat \G_\ell \mid Z(\sigma) = 0\},\\
\G_m &:= \textnormal{Im} \ \rho_m, & \G_m^t &:=  \{\sigma \in  \G_m \mid Z(\sigma) = 0\}.
\end{align*}

Define  for each positive integer $m$
\begin{align*}
	F_{m}& :=  m \frac{ \# \G_m^t}{\# \G_m}.
\end{align*}
\begin{lemma}\label{lemma:chebdens}
	Let $m$ be a positive integer then
	$$P_{m}(x) \sim \frac{1} m F_{m}.$$
\end{lemma}
\begin{proof}
	Let $m$ be a positive integer. Then by Chebotarev's density theorem   the conjugacy class of $\frob p$  is equidistributed in $\G_m$ where $p$ varies over all  primes  not dividing $m N$, i.e., for all  conjugacy class   $Cl$  of   $\G_m$ we have 
 $$\# \{p<x \textnormal{ prime }  \mid p\nmid mN \text{ and } \rho_{m} (\frob {p}) \in Cl \}\sim \frac {\# 	Cl} {\#  \G_m }  \pi(x).$$ 
	Denote by $\imath$ the morphism  of $\O_f$ to $ \O_m$ given by \ctext{sending $z$ to $ z\otimes 1$}.  If $p$ is a prime that does not divide $mN$, then  $\tr (\rho_m( \frob p)) = \imath (a_p(f)) \in \O_m$  (cf.~\cite[Theorem~3.1.a]{DDT}). Note that $Z(\sigma)$ only depends on the  conjugacy class of the matrix $\sigma$ so  $Z(\rho_{m}( \textnormal{Frob}_p))$ is well defined. Moreover $$\imath(Z_p) = Z(\rho_m(\textnormal{Frob}_p)) \in \O_m,$$
	since $Z_p = a_p(f) - \overline{a_p(f)}$ by definition. 	
	In particular  $Z_p \equiv 0 \mod m\sqrt D \Z$ if and only if $Z(\rho_m(\frob p )) =0$ as an element of $\O_m$. Hence
	\begin{align*}	
		P_{m}(x)& = \frac{ \# \{  p<x \textnormal{ prime}  \mid Z_{p} \equiv  0 \mod m\sqrt D\Z \}}{\pi(x)}\\ & \sim \frac{ \# \{  p<x \textnormal{ prime }  \mid p\nmid mN \textnormal{ and } Z( \rho_m (\frob p)) =0 \}}{\pi(x)} \\
		&\sim   \frac{ \# \{ \sigma \in \G_m  \mid Z(\sigma) = 0 \}}{\# \G_m}\\
		& = \frac{\#\G_m^t}{\#\G_m}. \qedhere
	\end{align*} 
\end{proof}

The following theorem will enable us to give explicit formulas for the cardinalities of $\G_{\ell^k}$ and  $\{\G_{\ell^k} \mid Z(\sigma) = 0\}$ for almost all primes.

Define
$$\widehat\A_\ell = \{\sigma \in GL_2(\widehat\O_\ell) \mid \det \sigma \in \Z_\ell^\times \}.$$
\begin{theorem}[Ribet]\label{theorem:largeimage}
	 \ctext{Let $f$ be a weight $2$ cusp form without inner twists}. Then for all primes $\ell$ the image of the $\ell$-adic representation, $\widehat\G_\ell$, is an open subgroup of $\widehat\A_\ell$.
	 Moreover  $\widehat\G_\ell = \widehat\A_ \ell$ for almost all primes. We say that the prime $\ell$ has large image if the inclusion is an equality, and we say that  $\ell$ is exceptional otherwise.
\end{theorem}
\begin{proof}
	This is a special case of \cite[Theorem 0.1]{RibetI}.
\end{proof}
Let $\ell$ be a prime and $k$ a positive integer. Define  $\A_{\ell^k}$ be the image of $\widehat\A_\ell$ under the natural projection modulo $\ell^k$ and $\A_{\ell^k}^t = \{\sigma \in \A_{\ell^k} \mid Z(\sigma) = 0\}$. If $\ell$ is a prime with large image, then $\A_{\ell^k} = \G_{\ell^k}$ so $$F_{\ell^k} = \ell^k\frac{\#\A_{\ell^k}^t}{\# \A_{\ell^k}}.$$ Note that $\widehat\O_\ell \cong \Z_\ell \times \Z_\ell$ if  $\ell$ splits in $\O_f$. If $\ell$ is  inert in $\O_f$ then $\widehat \O_\ell$ is the ring of integers of the unique unramified quadratic extension of $\Q_\ell$, denoted  by $\Z_{\ell^2}$. 
 In the next two sections we describe the cardinalities of  $\A_{\ell^k}$ and  $\A_{\ell^k}^t$ in the inert and split case respectively. For both cases we will need the following lemma and its corollary.
\begin{lemma}\label{lemma:carGl2}
	Let $R$ be a finite local ring with maximal ideal $\m$. Denote $r = \#R$ and $m = \# \m$. Then
	$$ \# GL_{2}(R) = (r^{2}-m^{2})\cdot r \cdot (r-m).$$ 
\end{lemma}
\begin{proof}
	Let $\sigma = \begin{pmatrix} a & b \\ c & d \end{pmatrix} \in M_{2\times 2}(R)$. Then $$\sigma \in GL_{2}(R) \Leftrightarrow ad \not \equiv bc \mod \m.$$ The vector $(a,b)$ can be any vector not contained in $\m ^{2}$. There are $r^{2}-m^{2}$ such vectors. We consider two cases depending on the valuation of $b$ in $R$. 
	
	First, if $b\in \m$, then $a \in R^{\times}$ so
	\begin{align*}
		ad \not \equiv bc \mod \m & \Leftrightarrow d  \not \equiv bc a^{-1} \mod \m \\
		& \Leftrightarrow d \not \equiv 0 \mod \m.
	\end{align*}
	Hence $c$ can be any element of $R$ and $d$ any element of $R^{\times}$. There are $r\cdot (r-m)$ such vectors $(c,d)$.
	
	Second, if $b\not \in \m$, then 
	\begin{align*}
		ad \not \equiv bc \mod \m & \Leftrightarrow adb^{-1}  \not \equiv c \mod \m. 
	\end{align*}
	Hence  $d$ can be any element of $R$ and $c$ any element not contained in $adb^{-1}+\m$. There are $r\cdot (r-m)$ such vectors.
	
	In either case we obtain $r\cdot (r-m)$ possibilities for the second vector. Hence $\#GL_{2}(R) = (r^{2}-m^{2})\cdot r\cdot (r-m)$.
\end{proof}
\begin{corollary}\label{cor:carGl2lk}
Let $\ell$ be a prime and $k$ a positive integer. Then
$$\#GL_2(\Z/\ell^k\Z) = \ell^{4k-3}(\ell^2-1)(\ell-1).$$
\end{corollary}

\subsection{Inert Primes}
Let $f$ be as  above, suppose that $\ell$ is an odd inert prime in $\O_f$ the ring of integers of the coefficient field of $f$. Then 
$$\widehat\O_\ell \cong \O_f \otimes_\Z \Z_\ell \cong \Z_{\ell^2}.$$  Recall that $\Z_{\ell^2}$ is the ring of integers of the unique unramified quadratic extension of $\Q_\ell$.
If $\alpha \in \overline \Q$ with $\alpha^2$ a square-free integer that is congruent to a quadratic non-residue modulo $\ell$, then $\Q_\ell(\alpha)$ is an unramified quadratic extension of $\Q_\ell$ \ctext{with ring of integers $\Z_\ell[\alpha]$} hence
$$\Z_{\ell^2} \cong \Z_\ell [\alpha].$$
Moreover  the morphism $\overline{\, \cdot \,}$ is given by
$$\overline{\, \cdot \,}: \Z_\ell [\alpha] \rightarrow \Z_\ell [\alpha]: a+\alpha b \mapsto a - \alpha  b.$$

If $\ell$ is an inert prime in $\O_f$, then we can take $\alpha = \sqrt D$ with $D$ the square-free integer such that $K_f=\Q(\sqrt D)$.  If moreover the $\ell$-adic representation attached to $f$ has large image in the sense of Theorem \ref{theorem:largeimage} \ctext{the sets $ \G_{\ell^k}$ and  $ \{ \sigma \in \G_{\ell^k}  \mid Z(\sigma) = 0 \} $} are respectively 
\begin{align*}
	\A_{\ell^k,I}   &:= \big\{ \sigma \in GL_2(\Z[\alpha] /\ell^k\Z[\alpha]) \mid \det \sigma \in \Z/\ell^k\Z^\times \big\}\textnormal{ and } \\
	\A_{\ell^k,I}^t &:= \big\{ \sigma \in GL_2(\Z[\alpha] /\ell^k\Z[\alpha]) \mid \det \sigma \in \Z/\ell^k\Z^\times \wedge  \ \tr \sigma \in \Z/\ell^k\Z \big\}.
\end{align*}

\begin{proposition}\label{prop:GLKInert}
Let $\ell$ be an odd prime and $k$ a positive integer. Then
$$\#\A_{\ell^k,I} = \ell^{7k-5}(\ell^4-1)(\ell- 1).$$
\end{proposition}
\begin{proof}
 There are $ (\ell^2-1) \ell^{2k-2}$ units $\Z[\alpha]/\ell^k\Z[\alpha]$, of which $(\ell- 1)\ell^{k-1}$  units are embedded in $\Z/\ell^k\Z$. By Lemma \ref{lemma:carGl2} 
$$\#GL_2(\Z[\alpha]/\ell^k \Z[\alpha]) = \ell^{8k-6}(\ell^4-1)(\ell^2-1).$$
Since any unit occurs equally many times as the determinant of a  matrix in $GL_2(\Z[\alpha]/\ell^k\Z[\alpha])$ we obtain
\begin{align*}
\#\A_{\ell^k,I} & = \#GL_2(\Z[\alpha]/\ell^k \Z[\alpha]) \frac{\#\Z/\ell^k\Z^\times}{\#\Z[\alpha]/\ell^k\Z[\alpha]^{\ \times}} \\
& =  \ell^{8k-6}(\ell^4-1)(\ell^2-1) \frac{(\ell- 1)\ell^{k-1}}{(\ell^2-1) \ell^{2k-2}} \\
& = \ell^{7k-5}(\ell^4-1)(\ell-1). \qedhere
\end{align*}
\end{proof}

\begin{proposition}\label{prop:GLKTInert}
Let $\ell$ be an odd prime and $k$ a positive integer. Then
$$\#\A_{\ell^k,I}^t =\frac{(\ell-1)}{(\ell+1)}\ell^{6k-2}\left(\ell^2+\ell +1-\ell^{-2k} \right).$$ 
\end{proposition}
\begin{proof}
See Appendix \ref{app:GLKTInert}.
\end{proof}

\subsection{Split Primes}
Let $f$ be as above, suppose that $\ell$ is an odd split prime in $\O_f$. Then 
$$\widehat\O_\ell \cong \Z_\ell \times \Z_\ell$$ and the morphism induced by the unique non-trivial element of the Galois group of $K_f$ over $\Q$ by the tensor product over $\Z$ with $\Z_\ell$ is 
$$\overline{\, \cdot \,} : \Z_\ell \times \Z_\ell \rightarrow \Z_\ell \times \Z_\ell: (a,b) \mapsto (b,a).$$
In particular the embedding of $\Z$ into $\widehat\O_\ell$ is diagonal. Suppose that $\ell$ has large image in the sense of Theorem \ref{theorem:largeimage}. Then the image of the Galois representation modulo $\ell^k$ and its subset \ctext{$\{\sigma \in \G_{\ell^k} \mid \tr \sigma \in \Z/\ell^k\Z \}$} are respectively
\begin{align*}
\A_{\ell^k,S} &:= \{(\tau, \tau') \in GL_{2}(\Z/\ell^k \Z)^2 \mid   \det \tau = \det \tau' \} \textnormal{ and} \\
\A_{\ell^k,S}^t & := \{(\tau, \tau') \in GL_{2}(\Z/\ell^k \Z)^2 \mid   \textnormal{char. poly.}\ \tau = \textnormal{char. poly.}\ \tau' \}, 
\end{align*}  
where char. poly. $\tau = X^2 - \tr \tau X + \det \tau$ denotes the  characteristic polynomial of  $\tau$.

\begin{proposition}\label{prop:GLKSplit}
Let $\ell$ be an odd prime and $k$ a positive integer. Then 
$$\#\A_{\ell^k,S} = \ell^{7k-5}(\ell^2-1)^2(\ell-1).$$
\end{proposition}
\begin{proof}
The determinant is equidistributed  in the units of $(\Z/\ell^k\Z)^2$.  By Corollary \ref{cor:carGl2lk} $\#GL_2(\Z/\ell^k\Z)^2 = \big(\ell^{4k-3}(\ell^2-1)(\ell-1)\big)^2$. Moreover there are $\big(\ell^{k-1}(\ell-1)\big)^2$ units in $(\Z/\ell^k\Z)^2$ and $\ell^{k-1}(\ell-1)$ are contained in $\Z/\ell^k \Z$. Hence 
\begin{align*}
\#\A_{\ell^k,S} & =\#GL_2(\Z/\ell^k\Z)^2 \frac{\# \Z/ \ell^k\Z^\times}{\#(\Z/\ell^k\Z^\times)^2} \\
&= \big(\ell^{4k-3}(\ell^2-1)(\ell-1)\big)^2 \frac{ \ell^{k-1}(\ell-1)}{\big(\ell^{k-1}(\ell-1)\big)^2} \\
& = \ell^{7k-5}(\ell^2-1)^2(\ell-1). \qedhere
\end{align*}
\end{proof}
\begin{proposition}\label{prop:GLKTSplit}
Let $\ell$ be an odd prime and $k$ a positive integer. Then
$$\# \A_{\ell^k,S}^t = \frac{(\ell-1)}{(\ell+1)} \ell^{6k-4} \left( \ell^4+\ell^3-\ell^2-2 \ell- \ell^{-2k+2} \right).$$
\end{proposition}
\begin{proof}
See Appendix \ref{app:GLKTSplit}.
\end{proof}

\subsection{The limit of $F_m$}
In this section we describe the behaviour of the factor $F_m$ and its limit by divisibility.

\begin{lemma}\label{lemma:FlkFinte}
Let $\ell$ be  a prime. Then 
\begin{align*}
	\widehat F_\ell &:=  \lim_{k\rightarrow \infty} \ell^k\frac{\#\G_{\ell^k}^t}{\#\G_{\ell^k}} <\infty.
	\intertext{Moreover if $\ell$ is odd, unramified and has large image, }
	\widehat F_\ell &=
	\begin{dcases} 
		{ \frac{\ell^3(\ell^2+\ell+1)}{(\ell+1)(\ell^4-1)} } & \textnormal{ if } \ell \textnormal{ is inert} \\ { \frac{\ell^2(\ell^3+\ell^2 -\ell-2)}{(\ell^2-1)^2 (\ell+1)} }&\textnormal{ if  } \ell \textnormal{ is split.}   
	\end{dcases}
\end{align*}
\end{lemma}
\begin{proof}
First we show that $\widehat F_\ell$ is finite if $\ell$ is a prime with large image by deducing an upper bound on $\# \A_{\ell^k}^t$ and a lower  bound on $\#\A_{\ell^k}$.

Let $D$ be the positive square-free integer such that $K_f = \Q(\sqrt D)$. Then $\O_f = \Z[X]/(X^2 -D)$ and 
$$\O_ {\ell^k}  \cong (\Z/\ell^k \Z[X])/(X^2 -D).$$
In particular we can embed $\A_{\ell^k }^t$ into 
$$M_{\ell^k}^t :=\Big\{ \sigma \in M_{2\times 2}\big((\Z/\ell^k \Z[X])/(X^2-D)\big) \mid \tr \sigma \in \Z/\ell^k\Z \textnormal{ and } \det \sigma  \in \Z/\ell^k \Z \Big\}.$$
We deduce an upper bound on the size of the latter set. Let $a_1$, $a_2$, ...  and $d_2$ be elements of $ \Z/\ell^k \Z$ and  $\sigma = \begin{pmatrix} a_1 + Xa_2 & b_1 + Xb_2 \\ c_1 + Xc_2 & d_1 + Xd_2\end{pmatrix}$ then 
\begin{align*}
\tr \sigma \in \Z/\ell^k\Z & \Leftrightarrow a_2+d_2 = 0 \textnormal{ and}\\
\det \sigma \in \Z/\ell^k \Z & \Leftrightarrow a_1d_2 +a_2d_1 - b_1c_2 - b_2c_1 = 0.
\end{align*}
In particular there is a bijection of sets between  $ M_{\ell^k}^t$ and the 7-tuples $(a_1,a_2,...,d_1)$ satisfying
\begin{equation}\label{eq:Mlk}
 a_1a_2 - a_2d_1 +b_1c_2 + b_2c_1 = 0.
\end{equation}
Denote by $\v$ the $\ell$-adic valuation on $\Z/\ell^k \Z$. Let $t=\textnormal{min}\{k,\v(a_2),\v(b_2),\v(c_2)\}$. If $t= k $, then \eqref{eq:Mlk} holds for  any $a_1$, $b_1$, $c_1$ and $d_1$ in $\Z/\ell^k\Z$. So there are at most $\ell^{4k}$ matrices in $M_{\ell^k}^t$ with $a_2 = b_2 = c_2 = 0$.

Suppose that  $ t$ is strictly smaller than $k$. Take $a_2'$, $b_2'$ and $c_2'$   such that $a_2 \ell^t= a_2'\ell^t$, $b_2 = b_2'\ell^t$ and $c_2 = c_2'\ell ^t$. By construction at least one of $a_2'$, $b_2'$ or $c_2'$ is invertible in $\Z/\ell^{k-t}\Z$. Suppose that  $a_2'$ is a unit. Then
\begin{align*}
\eqref{eq:Mlk} \Leftrightarrow d_1 \equiv a_2'^{-1}\big(a_1a_2'+b_1c_2'+b_2'c_1 \big) \mod \ell^{k-t}. 
\end{align*}
Hence for every $0\le t < k$ there are at most $\ell^{6k-2t-1}(\ell-1)$ matrices in $M_{\ell^k}^t$ with $\v(a_2) = t$, $\v(b_2)\ge t$ and  $\v(c_2)\ge t$. If $b_2'$ or $c_2'$ is invertible we  obtain at most $\ell^{6k-2t-1}(\ell-1)$ matrices in $M_{\ell^k}^t$ by a similar argument. So for every $t<k$ there are at most $3\ell^{6k-2t-1}(\ell-1)$ matrices in $M_{\ell^k }^t$ with $t=\textnormal{min}\{\v(a_2),\v(b_2),\v(c_2)\}$. Summing over all $t$  yields
\begin{align*}
\#\A_{\ell^k}^t & \le \#M_{\ell^k}^t \le \ell^{4k} + 3 \sum_{t=0}^{k-1} \ell^{6k-2t-1}(\ell-1) \\
& = \ell^{4k} +3(\ell- 1)\ell^{4k-1}\big(\ell^{2k} + \ell^{2k-2} + \cdots + \ell^2 \big) \\
& = \ell^{4k}\left(1+3(\ell-1)\ell\frac{(\ell^{2k} - 1)}{(\ell^2-1)} \right) \\
& \le 3\ell^{6k}.
\end{align*}

Next we deduce a lower bound on $\#\A_{\ell^k}$. Let $u$ and $v$ be  units in $\Z/\ell^k \Z$,$b$ and $c$ elements of $(\Z/\ell^k\Z[X])/(X^2 -D)$ and  $m \in \ell\Z/\ell^k\Z$. Then $a = (v+mX)$ is a unit in $(\Z/\ell^k\Z[X])/(X^2 - D)$. Indeed the inverse is given by $(v^2 -m^2 D)^{-1}(v-mX)$. If $d = (u+bc)a^{-1}$ the matrix  $\sigma = \begin{pmatrix} a & b \\c & d \end{pmatrix}$   has determinant $u$ so belongs to $\A_{\ell^k}$.   In particular there are at least $\ell^{7k-3}(\ell-1)^2$ elements in $\A_{\ell^k}$. Using the upper bound on $\#\A_{\ell^k}^t$ and the lower bound on $\#\A_{\ell^k}$ we obtain
$$\widehat F_\ell = \lim_{k\rightarrow \infty} \ell^k \frac{\#\A_{\ell^k}^t}{\#A_{\ell^k}} \le\lim_{k\rightarrow \infty} \ell^k \frac{3\ell^{6k+1}}{\ell^{7k-3}(\ell^2-1)} \le 4\ell^2.$$

If  $\ell$ is an exceptional prime, then $\widehat\G_\ell$ is an open subgroup of $\widehat\A_\ell$ by Theorem \ref{theorem:largeimage}. So that 
	\begin{align*}
		\lim_{k\rightarrow \infty} \frac{\#\A_{\ell^k}}{\#\G_{\ell^k}}  & = \# \left(\widehat\A_\ell/\widehat\G_\ell \right)< \infty. \\
		\intertext{Moreover, $\G_{\ell^k}^t \subset \A_{\ell^k}^t$ so}
		\widehat F_\ell = \lim_{k\rightarrow \infty} \ell^k\frac{\#\G_{\ell^k}^t}{\#\G_{\ell^k}} &\le \lim_{k\rightarrow \infty} \ell^k\frac{\#\A_{\ell^k}^t}{\#\G_{\ell^k}} \\
		& = \lim_{k\rightarrow \infty} \ell^k \frac{\#\A_{\ell^k}^t}{\#\A_{\ell^k}} \cdot \lim_{k\rightarrow \infty} \frac{\#\A_{\ell^k}} {\#\G_{\ell^k}} \\
		&< \infty.
	\end{align*}
Finally from  Propositions \ref{prop:GLKInert} and \ref{prop:GLKSplit}  we obtain
	\begin{align*}
		\#\A_{\ell^k} & = \begin{cases} \ell^{7k-5}(\ell^4-1)(\ell-1) & \textnormal{ if } \ell \textnormal{ is inert} \\ \ell^{7k-5}(\ell^2-1)^2(\ell-1)  &\textnormal{ if  } \ell \textnormal{ is split.}   \end{cases}
		\intertext{Moreover from Propositions \ref{prop:GLKTInert} and \ref{prop:GLKTSplit}}
		\#\A_{\ell^k}^t & = \begin{cases} \frac{\ell-1}{\ell+1} \ell^{6k-2}\left
		( \ell^2 + \ell + 1 - \ell^{-2k}\right) & \textnormal{ if } \ell \textnormal{ is inert} \\ \frac{\ell-1}{\ell+1} \ell^{6k-4}\left(\ell^4 +\ell^3 - \ell^2 - 2\ell - \ell^{-2k+2}\right)  &\textnormal{ if  } \ell \textnormal{ is split.}   \end{cases}
		\intertext{So that}
		\lim_{k \rightarrow \infty} \ell^k \frac{\#\A_{\ell^k}^t}{\#\A_{\ell^k}} & =\begin{dcases} \frac{\ell^3(\ell^2+\ell+1)}{(\ell+1)(\ell^4-1)}  & \textnormal{ if } \ell \textnormal{ is inert} \\ \frac{\ell^2(\ell^3+\ell^2 -\ell-2)}{(\ell^2-1)^2 (\ell+1)} &\textnormal{ if  } \ell \textnormal{ is split.}   \end{dcases} \qedhere
	\end{align*}
\end{proof}

\begin{corollary}
Let $f$ be as above. Then the  limit of $F_m$ by divisibility exists, i.e.,
$$\widehat F := \limd{m} F_m < \infty.$$
\end{corollary}
\begin{proof}
First note that  for any sequence $a_n$ 
\begin{align*}
\sum_n a_n < \infty & \Rightarrow \sum_n \log ({a_n+1})  < \infty \\
& \Leftrightarrow \prod_n (a_n+1) < \infty. 
\end{align*}
Moreover if $\ell$ is an odd unramified prime with large image,    Lemma \ref{lemma:FlkFinte} yields 
\begin{align*}
\widehat F_\ell  & = \begin{dcases}1 + \frac{\ell^3 +\ell + 1}{(\ell+1)(\ell^4-1)}  & \textnormal{ if } \ell \textnormal{ is inert} \\ 1 + \frac{\ell^3 - \ell -1}{(\ell^2-1)^2 (\ell+1)} &\textnormal{ if  } \ell \textnormal{ is split.}   \end{dcases} 
\end{align*}
In particular the product $\prod \widehat F_\ell$ taken over all odd unramified primes with large image is finite. Since almost all primes are odd, are  unramified and have large image the product taken over all  primes is finite by Lemma \ref{lemma:FlkFinte}.
Finally  Serre's adelic open image theorem \cite[Theorem~3.3.1]{Loef} states that $\widehat \G$ is an open subgroup of $\prod_\ell \widehat \G_\ell$ hence
\begin{align*}
\limd m F_m & = \limd m m \frac{\#\G_m^t }{\# \G_m} \\
& \le \limd m \frac{\prod_{\ell^k \| m} \ell^k\#\G^t_{\ell^k} }{\#\G_m} \\
& = \limd m \frac{\prod_{\ell^k \| m} \ell^k \#\G^t_{\ell^k} }{\prod_{\ell^k \| m} \#\G_{\ell^k}} \cdot  \frac{\prod_{\ell^k \| m} \#\G_{\ell^k}}{\#\G_m} \\
& = \limd m \prod_{\ell^k \| m} \ell^k\frac{ \#\G^t_{\ell^k} }{ \#\G_{\ell^k}} \cdot \limd m \frac{\#\prod_{\ell^k \| m} \G_{\ell^k}}{\#\G_m} \\
&  = \prod_\ell \widehat F_\ell \cdot \#\left(\scriptsize{\prod_{\ell}}\widehat \G_{\ell}\big/ \widehat\G\right)\\
& < \infty. \qedhere
\end{align*}
\end{proof}

\section{Main Result} \label{section:MainResult}

In this section we state and prove our main result.

\begin{lemma}\label{lemma:MainResult}
Let $f$ be a weight $2$  normalized cuspidal Hecke eigenform of level $\Gamma_1(N)$, with quadratic coefficient field without inner twist. \ctext{Let $m$ be a positive integer such that Assumptions \ref{ass:PmAss} and \ref{ass:ind1}  hold}. Then  
 $$\#\{p<x \textnormal{ prime} \mid a_p \in \Q \} \sim \frac{F_m}{\alpha_m} \frac{16 \sqrt D }{3\pi^2}\frac{ \sqrt x}{\log x}, $$
 with $0<\alpha_m<\infty$ as in Assumption \ref{ass:ind1}.
 \end{lemma}
 \begin{proof}
Denote $N_f(x) = \#\{p <x \textnormal{ prime }\mid a_p(f) \in \Q\} = P(x)\cdot \pi(x)$ and let $\varepsilon>0$. By Assumption \ref{ass:ind1} there exists  $x_1>0$  such that for all $x>x_1$
$$ \left|  \frac{P^m(x)P_m(x)}{\alpha_m P(x)}-1 \right|<\varepsilon/6. $$
  Let $x_2$ be such that for all $x>x_2$ we have
$$\left| \frac{F_m}{mP_m(x)} - 1 \right|<\frac{\varepsilon} 6.$$
Such an $x_2$ exists by Lemma \ref{lemma:chebdens}. 
By Corollary \ref{cor:PmCor} under Assumption \ref{ass:PmAss} there exists an $x_3$ such that for all $x>x_3$
$$\left| \frac{16m\sqrt D}{3\pi^2 \sqrt x P^m(x)} - 1 \right|<\frac{\varepsilon} 6.$$
Finally let $x_4$ be such that for all $x>x_4$
$$\left| \frac{x}{\pi(x) \log(x)} -1\right|<\frac{\varepsilon}{6}.$$
Then for any $x >\textnormal{max}\{ x_1,x_2,x_3,x_4\}$ we obtain
\begin{align*}
\left|\frac{1}{\alpha_m}\frac{16\sqrt D F_m \sqrt x}{3\pi^2\log x N_f(x)}-1 \right|& =\left|\frac{P^m(x)P_m(x)}{\alpha_m P(x)} \cdot \frac{16\sqrt D m }{3\pi ^2 \sqrt x P^m(x)} \cdot \frac{F_m}{mP_m(x)} \cdot \frac{x}{\pi(x)\log x}-1\right|\\
& < \left|\left(1+\frac{\varepsilon} 6\right)^4 -1\right| \\
& <\varepsilon. \qedhere
\end{align*} 
\end{proof}

\begin{corollary}\label{cor:MainResult}
Let $f$ be as above. \ctext{Suppose that  there exists a positive integer $m_0$ such that Assumptions \ref{ass:PmAss} and \ref{ass:ind1} hold for $m_0$}. 
\begin{enumerate}
\item Then Assumption \ref{ass:PmAss} implies \ref{ass:ind1} for any positive integer $m$.
\item If Assumption \ref{ass:PmAss} is true for all positive integers  $m \in M\Z$  for some $M$, then 
$$0<\limd m \lim_{x \rightarrow \infty} \frac{P^m (x)\cdot P_m(x)}{P(x)} =:\alpha< \infty.$$
Moreover
$$\#\{p  < x \textnormal{ prime } \mid a_p(f) \in \Q\} \sim \frac{\widehat F}{\alpha} \frac{16 \sqrt D} {3\pi ^{2} }\frac{\sqrt x}{\log x}.$$
\end{enumerate}
\end{corollary}
\begin{proof}
If such an $m_0$ exists then by Lemma \ref{lemma:MainResult}  there exists a positive non-zero constant  $C_{m_0}$ such that  $P(x) \sim C_{m_0}\sqrt x/\log x $.
\begin{enumerate}
\item Let $m$ be a positive integer satisfying Assumption \ref{ass:PmAss}. By Corollary \ref{cor:PmCor} and Lemma \ref{lemma:chebdens} $P^m (x)\cdot P_m(x) \sim C' \sqrt x/\log x $ with $0<C' = \frac{16 \sqrt D F_m}{3 \pi^2}<\infty$. Hence 
\begin{align*}
\lim_{x \rightarrow \infty } \frac{P^m (x)\cdot P_m(x)}{P(x)} & = \lim_{x \rightarrow \infty } \frac{C' \sqrt x /\log x}{C_{m_0} \sqrt x /\log x} \\
& = \frac{C'}{C_{m_0}}.
\end{align*}
\item By the first point  we can apply  Lemma \ref{lemma:MainResult} to any pair of  positive integers $m$ and $m'$ in $M\Z$ and obtain that
 $$ \frac{\alpha_m }{F_m} = \frac { \alpha_{m'}}{F_{m'}}.$$
 So the following definition of $\alpha$ does not depend on the choice of $m$.
 $$0< \alpha:= \frac{\alpha_m \widehat F}{F_m }< \infty.$$
 In particular 
 \begin{align*}
 \limd m \lim_{x \rightarrow \infty} \frac{P^m (x)\cdot P_m(x)}{P(x)}& = \limd m {\alpha_m} \\
  &= \limd m \frac{\alpha F_m} {\widehat F} \\
 & = \alpha. 
 \end{align*}
 By Lemma \ref{lemma:MainResult} we obtain for every positive integer $m$ divisible by $M$ that 
\begin{align*}
1 &= \lim_{x\rightarrow \infty} \frac{F_m16\sqrt D \sqrt x }{\alpha_m 3\pi^2 \log x N_f(x)}. \\
\intertext{Taking the limit by divisibility of $m$ yields}
1 & = \limd m  \lim_{x\rightarrow \infty} \frac{F_m16\sqrt D  \sqrt x }{\alpha_m 3\pi^2 \log x N_f(x)}. \\
\intertext{Since $F_m$ and $\alpha_m$ do not depend on $x$ we obtain}
1 & =  \limd m \frac{F_m}{\alpha_m} \lim_{x\rightarrow \infty} \frac{16\sqrt D  \sqrt x }{3\pi^2 \log x N_f(x)} \\
& = \lim_{x\rightarrow \infty} \frac{\widehat F16\sqrt D  \sqrt x }{\alpha 3\pi^2 \log x N_f(x)}. \qedhere
\end{align*}
\end{enumerate}
\end{proof}

\begin{theorem}\label{theorem:MainResult}
	\ctext{Let $f$ be a weight $2$  normalized cuspidal Hecke eigenform of level $\Gamma_1(N)$ with quadratic coefficient field $\Q(\sqrt D)$ and without inner twist. Suppose that  there exists a positive integer $m_0$ such that  Assumptions  \ref{ass:PmAss} and \ref{ass:ind} hold for $f$ and all positive integers in $m_0\Z$}.
	Then $$\#\{p  < x \textnormal{ prime } \mid a_p(f) \in \Q\} \sim  \frac{16 \sqrt D \widehat F} {3\pi ^{2}}\frac{\sqrt x}{\log x}.$$
\end{theorem}
\begin{proof}
From Corollary \ref{cor:MainResult} we obtain $$\#\{p  < x \textnormal{ prime } \mid a_p(f) \in \Q\} \sim \frac{\widehat F}{\alpha} \frac{16 \sqrt D} {3\pi ^{2} }\frac{\sqrt x}{\log x}.$$
The claim of Assumption \ref{ass:ind} is precisely that $\alpha = 1$. Hence the theorem follows.
\end{proof}

\section{Numerical Results}\label{sect:Num}
In this  section we provide numerical results that support the assumptions made and the results deduced from these assumptions. Moreover we describe the method used to obtain these results.

All computations are done using   the following six new Hecke eigenforms $f_N \in S_2(\Gamma_0(N))$:
\small
\begin{align*} 
f_{29}	& ={q+(-1+\sqrt 2)q^2 +(1-\sqrt 2) q^3 +(1-2\sqrt 2) q^4-q^5 + \cdots}\\
f_{43} & = q+\sqrt 2 q^2 - \sqrt 2 q^3 +(2-\sqrt 2 ) q^5 + \cdots\\
f_{55} & = q+(1+\sqrt 2)q^2  -2\sqrt 2q^3 +(1+2\sqrt 2 )q^4 -q^5 \cdots\\
f_{23} & = q+\frac{-1+\sqrt 5}{2} q^2- \sqrt 5 q^3 - \frac{1+\sqrt 5} 2 q^4 + (-1+\sqrt 5)q^5 + \cdots \\
f_{87} & =q + \frac{1+\sqrt 5} 2 q^2 + q^3 + \frac{-1+\sqrt 5} 2 q^4 + (1
-\sqrt 5)q^5 \cdots \\
f_{167} & = q+ \frac{-1+\sqrt 5} 2 q^2 -\frac{1+\sqrt 5} 2 q ^3  -\frac{1+\sqrt 5} 2 q ^4 - q^5 + \cdots.
\end{align*}
\normalsize
Note that the level for each of the eigenforms is square-free and the nebentypus is trivial so by Lemma \ref{lemma:prel0}  none of these eigenforms have inner twists. Moreover the coefficient field of $f_{29}$, $f_{43}$ and $f_{55}$ is $
\Q(\sqrt 2)$ and the coefficient field of $f_{23}$, $f_{87}$ and $f_{167}$ is $\Q(\sqrt 5 )$. In this section  we will denote $\A_{f_N}$, $c_{f_N}$,... by $\A_N$, $c_N$, ... respectively.

As described in Section \ref{section:suth} the Galois orbit of the $p$-th coefficient of a eigenform $f_N$ can be computed from the $L_p$-polynomial of the  abelian variety $\A_N$ associated to $f_N$. For each eigenform $f_N$ we give an equation for a hyperelliptic curve $C_N$ such that the Jacobian $J(C_N)$ is isomorphic to the abelian variety $\A_N$.  Obtaining such an equation is a non-trivial problem. For levels $29$, $43$ and $55$ the equations are found in \cite[page~42]{Bend} and  \cite[page 137]{Wils} for the remaining three. The equations are: 
\begin{align*}
C_{29}& :y^2 =x^6 - 2x^5 + 7x^4 - 6x^3 + 13x^2 - 4x + 8, \\
C_{43}& :y^2 =  -3x^6- 2x^5 + 7x^4 - 4x^3 - 13x^2 + 10x - 7, \\
C_{55}&:y^2 =-3x^6 + 4x^5 + 16x^4 - 2x^3 - 4x^2 + 4x - 3, \\
C_{23}& :y^2 = x^6 + 2x^5 - 23x^4 + 50x^3 - 58x^2 + 32x - 11, \\
C_{87}&:y^2 =-x^6 + 2x^4 + 6x^3 + 11x^2 + 6x + 3, \\
C_{167}&:y^2 =-x^6 + 2x^5 + 3x^4 - 14x^3 + 22x^2 - 16x + 7.
\end{align*}

Next we use Andrew Sutherlands {smalljac} algorithm described in \cite{KeSu} to compute the coefficients of the $L_p$-polynomial of each hyperelliptic curve $C_N$. This algorithm is implemented in C and is  available  at Sutherland's web page. 
With this method we are able to compute the Galois orbit of the coefficients of one eigenform for all primes up to $10^8$ in less than $50$ hours. All computations are done on a Dell Latitude E6540 laptop with Intel i7-4610M processor (3.0 GHz, 4MB cache, Dual Core). The processing of the data and the creation of the graphs was done using Sage Mathematics Software \cite{sage} on the same machine. The running time of this is negligible compared to the  smalljac algorithm.

\subsection{Murty's Conjecture}

First we  check Conjecture \ref{conj:Murty}. For each eigenform we plot the number of primes $p<x$ such that the $p$-th coefficient is a rational integer for 50 values of $x$ up to $10^8$. According to  this conjecture there exists a constant $c_N$ such that 
$$\#\{p<x \textnormal{ prime} \mid a_p(f_N) \in \Q\} \sim c_N \frac{\sqrt x}{\log x}.$$
To check the conjecture we approximate $c_N$ using least squares fitting.  Denote this estimate by $\widetilde c_N$. Figure \ref{fig:model} provides numerical evidence for the  behaviour of $N(x)$ and column 2 of Table \ref{table:FN} list the values of $\widetilde c_N$ found by least squares fitting.
\begin{figure}[h!]
\centering
\includegraphics[width=0.8\linewidth]{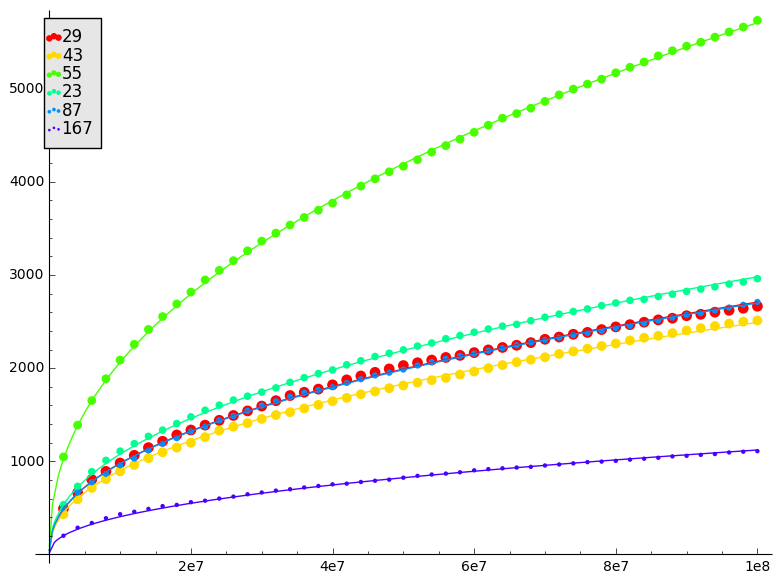}
\caption{Plot of  $\#\{p<x \textnormal{ prime} \mid a_p(f_N) \in \Q\}$ (dots) and  $\widetilde c_N \sqrt x/\log x$ (line) using  least squares fitting to compute $\widetilde c_N$ for $x$ up to $10^8$. }
\label{fig:model}
\end{figure}

\subsection{The place at infinity}
Corollary \ref{cor:PmCor} states that under Assumption \ref{ass:PmAss} and the generalized Sato-Tate conjecture $$\#\left\{p<x \textnormal{ prime} \mid Z_p \in ]-m\sqrt D/2, m\sqrt D/2[\right\} \sim \frac{16 \sqrt D m}{3\pi^2} \frac {\pi(x)} {\sqrt x}.$$
For  $m$ equal to $100$, $500$ and $1000$  Figure \ref{fig:Inf} indicates that $\frac{16m \sqrt D }{3\pi^2} \frac {\pi(x)} {\sqrt x}$ is in fact a good approximation for $\#\{p<x\textnormal{ prime}\mid Z_p \in ]-m\sqrt D/2, m\sqrt D/2[\}$. Although this neither proves Assumption \ref{ass:PmAss} nor the generalized Sato-Tate conjecture, it does confirm that $P_m(x)$  depends on the coefficient field of the eigenform. 
\begin{figure}[h!]
	\begin{subfigure}{0.329\textwidth}
		\includegraphics[width =  \linewidth ]{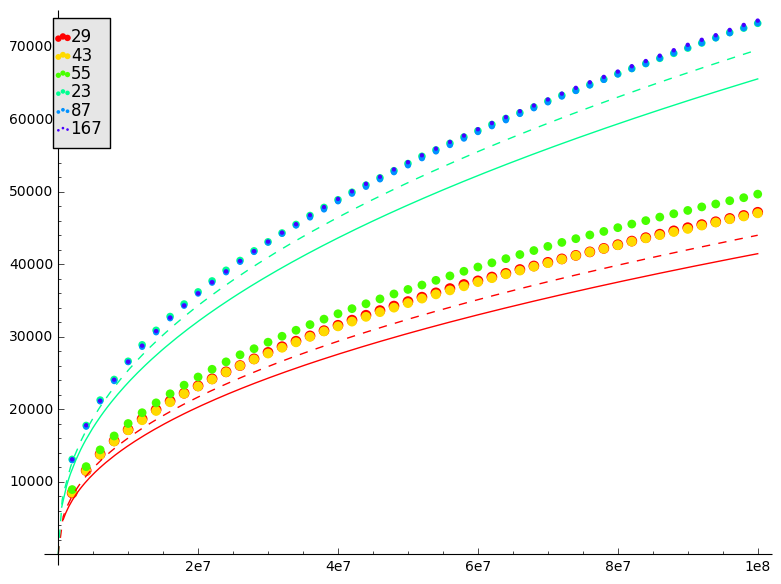}
		\caption{$m=100$}
	\end{subfigure}
	\begin{subfigure}{0.329\textwidth}
		\includegraphics[width =  \linewidth ]{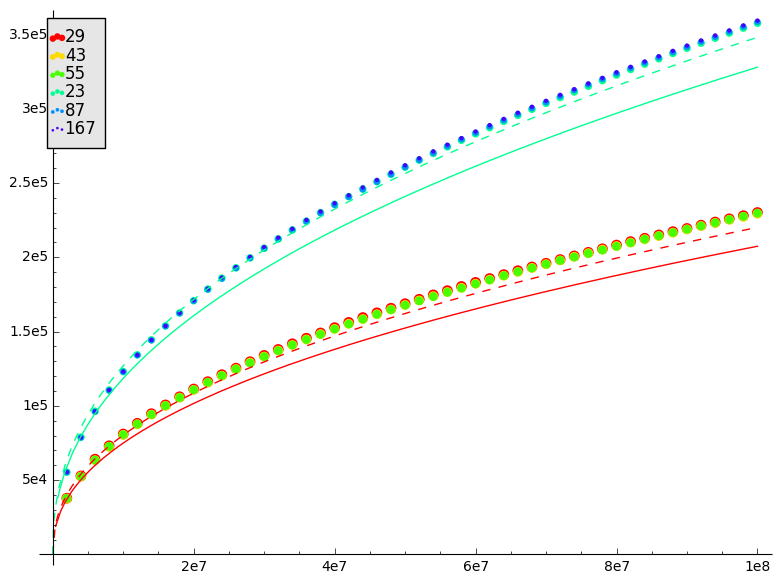}
		\caption{$m=500$}
	\end{subfigure}
	\begin{subfigure}{0.329\textwidth}
		\includegraphics[width =  \linewidth ]{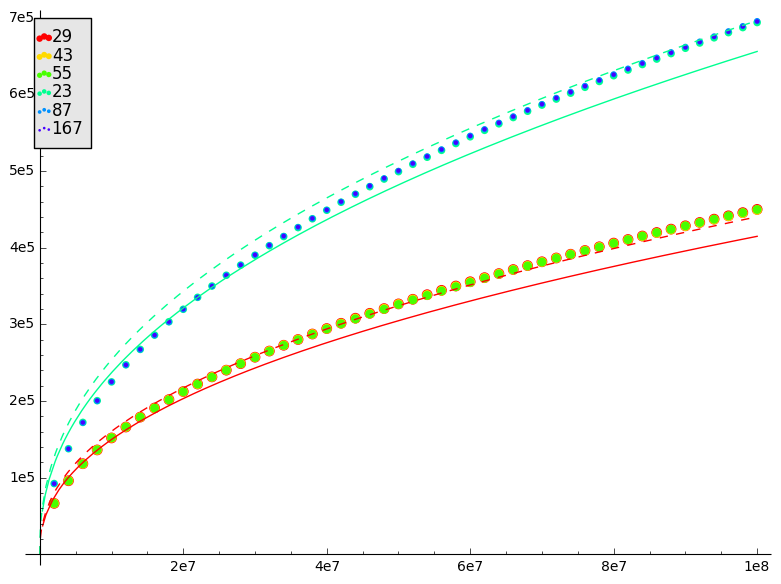}
		\caption{$m=1000$}
	\end{subfigure}
\caption{Plots of $\#\{p<x\textnormal{ prime} \mid Z_p \in ]-m\sqrt D/2, m\sqrt D/2[\}$ (dots),  $\frac{16 \sqrt D m}{3\pi^2} \frac {\pi(x)} {\sqrt x}$ (dashed) and $\frac{16 \sqrt D m}{3\pi^2} \frac {\sqrt x} {\log x}$ (full line) for $\sqrt D = \sqrt 5$ (cyan) and $\sqrt D = \sqrt 2$ (red) for each eigenform $f_N$ and $m =100$, $500$, $1000$.} 
\label{fig:Inf}
\end{figure}

\subsection{Finite Places}

In  \cite{BiDi} Nicolas Billerey and Luis Dieulefait provide explicit bounds on the primes $\ell$ that do not have large image for a given eigenform $f \in S_2(\Gamma_0(N))$ with square-free level $N$. In fact they provide a more general result. We only state the lemma for square-free level.
\begin{lemma}\label{cor:posExcP}
Let $f$ be a eigenform in $S_2(\Gamma_0(N))$. Assume that $N=p_1p_2\cdots p_t$, where $p_1,...,p_t$ are $t\geq 1$ distinct primes and $\ell$ is exceptional. Then $\ell$ divides  $15N$ or $p_i^2-1$ for some $1\leq i\leq t$.
\end{lemma}
\begin{proof}
This is the statement of \cite[Theorem~2.6]{BiDi} in the weight $2$ case.
\end{proof}
The eigenforms in our computations have weight $2$ and square-free level  so we can apply the lemma and obtain a list of primes that are possibly exceptional for each eigenform (Table \ref{table:FN}). 

If $\ell$ is an odd unramified prime with large image  Lemmas \ref{lemma:chebdens} and  \ref{lemma:FlkFinte}  yield
$$\#\{p<x\textnormal{ prime}\mid Z_p \equiv 0 \mod \ell^k\}\sim \pi(x)\cdot\begin{dcases} \frac{\ell^2 + \ell + 1 - \ell^{-2k}}{(\ell+1)(\ell^4-1)\ell^{k-3}} &\textnormal{if } \ell \textnormal{ is inert} \\ \frac{\ell^4+\ell^3 -\ell^2 -2\ell -\ell^{-2k+2}}{(\ell^2-1)^2\ell^{k-1} }& \textnormal{if } \ell \textnormal{ is split.}\end{dcases}$$

For each prime that is possibly exceptional and each eigenform we can confirm that the prime  is exceptional by comparing $\#\{p<x\textnormal{ prime}\mid Z_p \equiv 0 \mod \ell^k\}$ with the expected value for a prime with large image for $x$ up to $10^8$ (Fig. \ref{fig:mod}). For $\ell= 2$ we do not have a theoretic result for large image. Therefore none is plotted. The same holds for   $\ell=5$ and eigenforms  $f_{23}$, $f_{87}$ and $f_{167}$ since $5$ ramifies in $\Q(\sqrt 5)$. 

Some primes are inert in $\Q(\sqrt 5)$ and split in $\Q(\sqrt 2)$ or vice versa. So a priori we have  two possibilities for the behaviour of $\#\{p<x\textnormal{ prime}\mid Z_p \equiv 0 \mod \ell^k\}$ for a prime $\ell$ with large image. However the first prime for which this occurs is $7$. Indeed $7$ splits in $\Q(\sqrt 2)$ and is inert in $\Q(\sqrt 5)$. For $\ell = 7$ one can hardly distinguish the inert  and split case  visually.  

From  Figure \ref{fig:mod} we can confirm that  an odd unramified prime is exceptional for a given eigenform if the plot of $\#\{p<x \mid Z_p  \equiv 0 \mod \ell^k\}$ differs from that of the large image case. Moreover for any prime $\ell$ we can conclude that the image of the $\ell$-adic representation attached to different eigenforms is distinct. Note that the converse does not hold. Indeed the fact that two eigenforms exhibit the same behaviour with respect to $\#\{p<x \textnormal{ prime} \mid Z_p  \equiv 0 \mod \ell^k\}$ does not imply that their $\ell$-adic representations are the same.

For example if $\ell^k= 2$  (Fig. \ref{figure:mod2}), then all eigenforms except $f_{167}$  exhibit the same behaviour. But for $\ell^k=8$ (Fig. \ref{figure:mod8}) we clearly distinguish five different  representations. Note that we do not observe this behaviour  for any other prime. From $\ell^k=3$ (Fig. \ref{figure:mod3}) we can conclude that $3$ is an exceptional prime for the $3$-adic representation attached to $f_{43}$ and $f_{55}$. The primes that are marked in bold in the last column of Table \ref{table:FN} are   the primes for which Figure \ref{fig:mod} confirms the prime is exceptional.

\begin{figure}[h!]
	\begin{subfigure}{0.329\textwidth}
		\includegraphics[width =  \linewidth ]{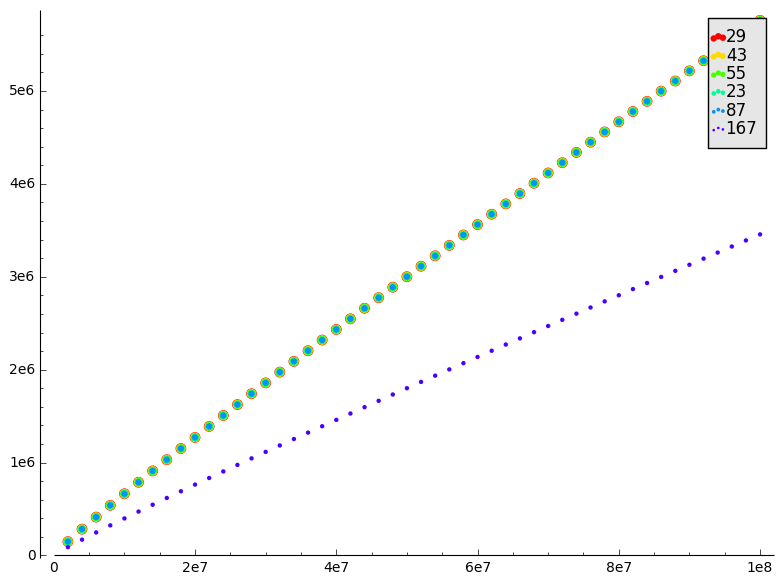}
		\caption{$\ell^k=2$}
		\label{figure:mod2}
	\end{subfigure}
	\begin{subfigure}{0.329\textwidth}
		\includegraphics[width =  \linewidth ]{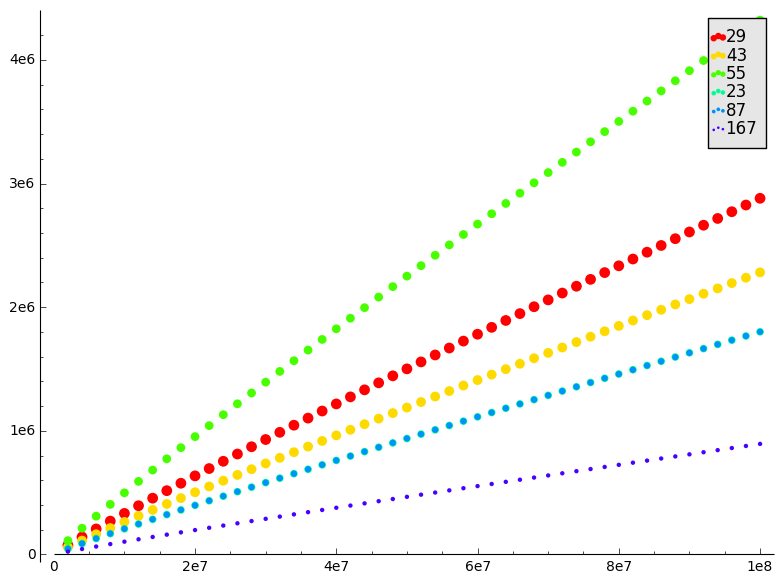}
		\caption{$\ell^k=8$}
		\label{figure:mod8}
	\end{subfigure}
	\begin{subfigure}{0.329\textwidth}
		\includegraphics[width =  \linewidth ]{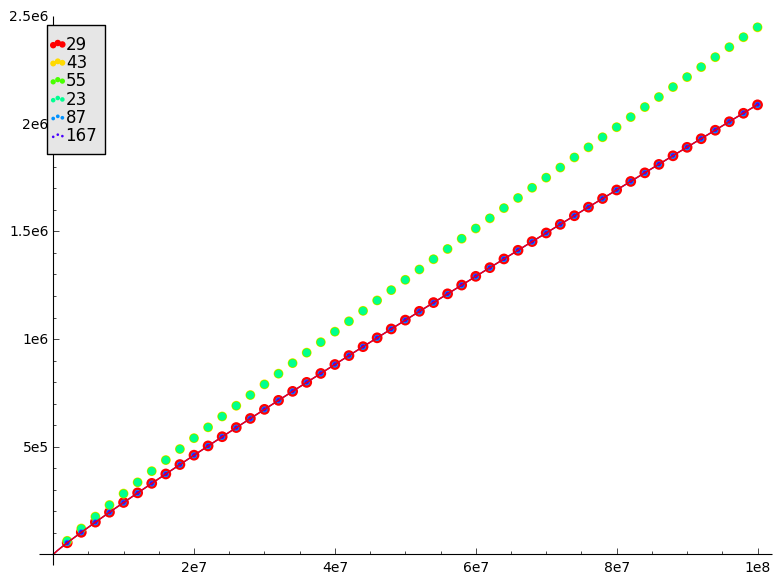}
		\caption{$\ell^k=3$}
		\label{figure:mod3}
	\end{subfigure}
	\begin{subfigure}{0.329\textwidth}\label{figure:mod5}
		\includegraphics[width =  \linewidth ]{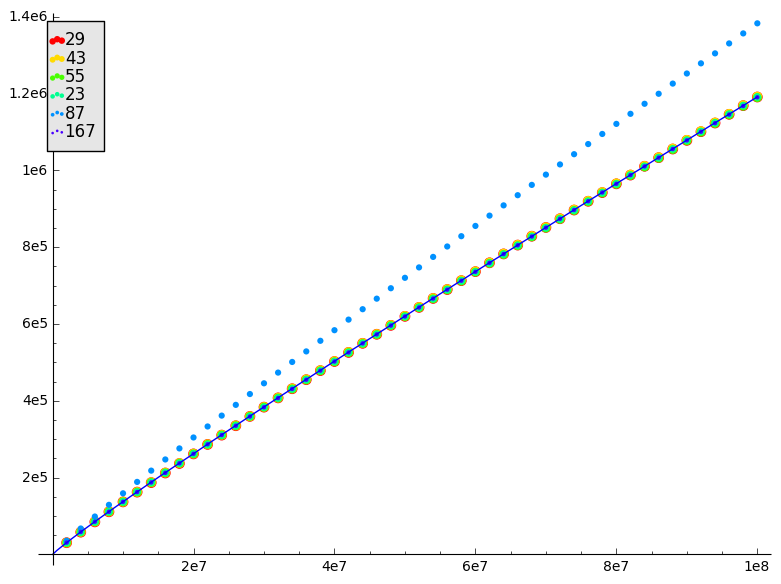}
		\caption{$\ell^k=5$}
	\end{subfigure}
	\begin{subfigure}{0.329\textwidth}\label{figure:mod7}
		\includegraphics[width =  \linewidth ]{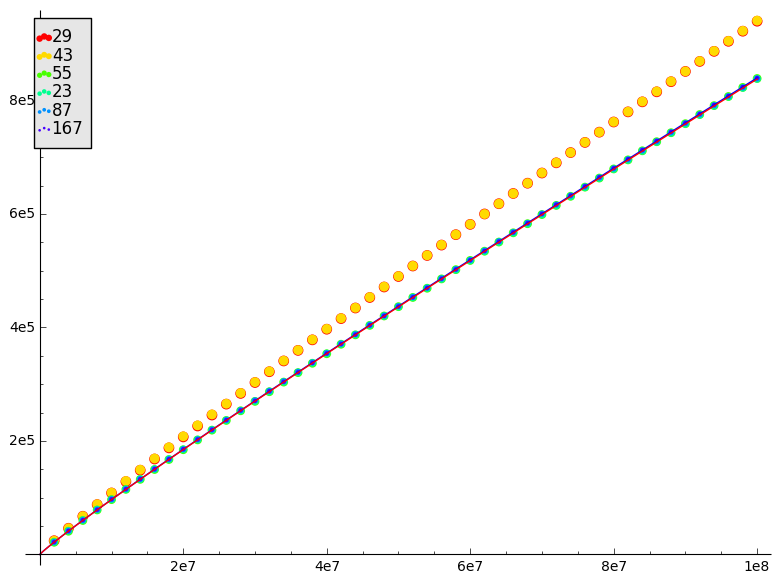}
		\caption{$\ell^k=7$}
	\end{subfigure}
	\begin{subfigure}{0.329\textwidth}\label{figure:mod11}
		\includegraphics[width =  \linewidth ]{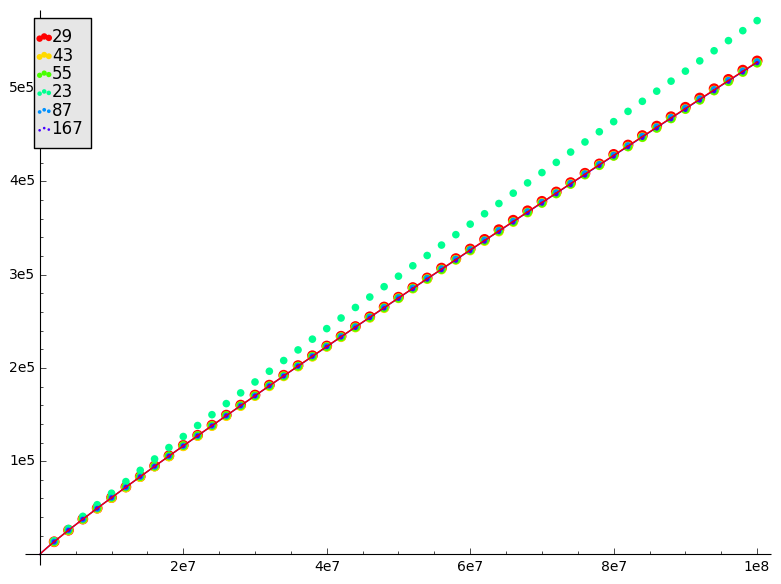}
		\caption{$\ell^k=11$}
	\end{subfigure}
	\begin{subfigure}{0.24\textwidth}\label{figure:mod29}
		\includegraphics[width =  \linewidth ]{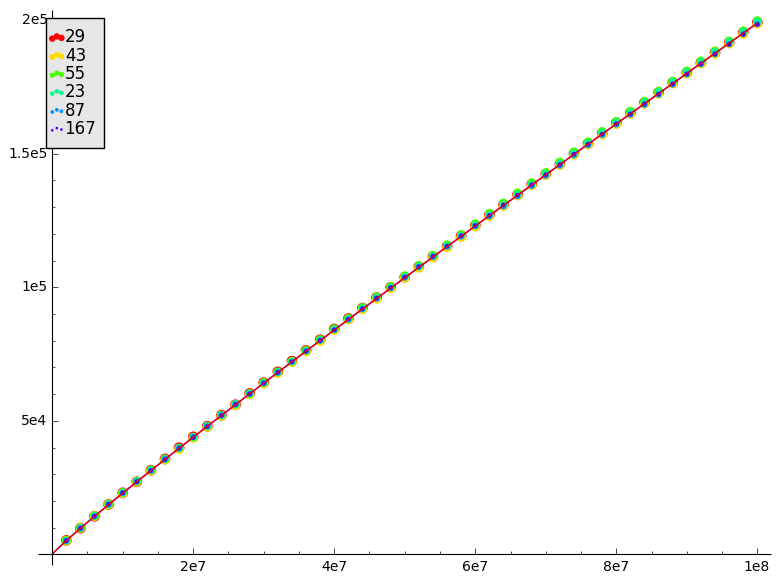}
		\caption{$\ell^k=29$}
	\end{subfigure}
	\begin{subfigure}{0.24\textwidth}\label{figure:mod43}
		\includegraphics[width =  \linewidth ]{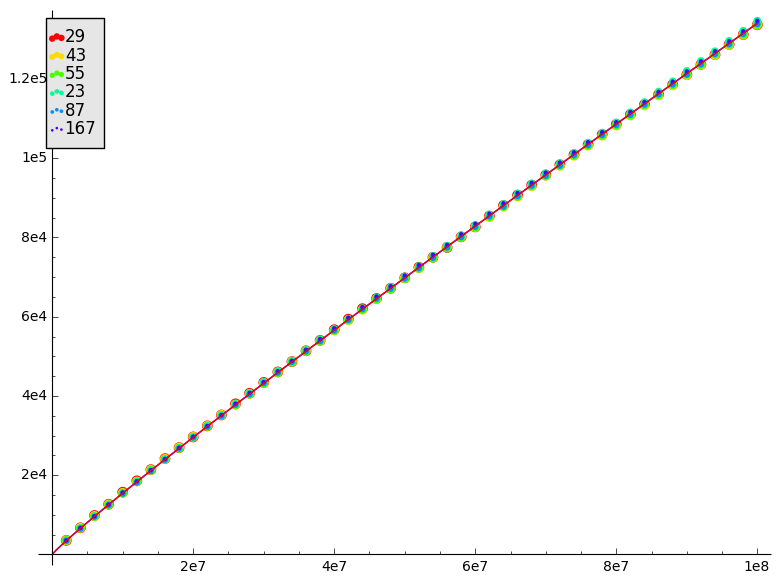}
		\caption{$\ell^k=43$}
	\end{subfigure}
	\begin{subfigure}{0.24\textwidth}\label{figure:mod83}
		\includegraphics[width =  \linewidth ]{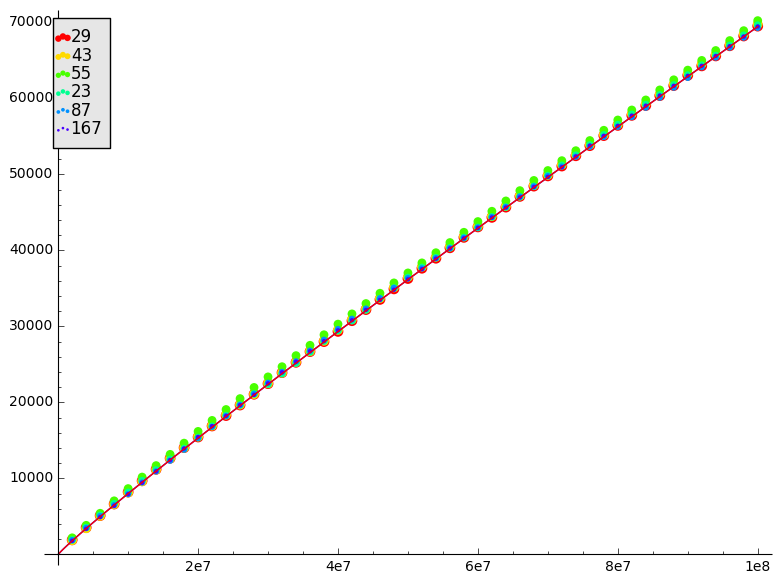}
		\caption{$\ell^k=83$}
	\end{subfigure}
	\begin{subfigure}{0.24\textwidth}\label{figure:mod167}
		\includegraphics[width =  \linewidth ]{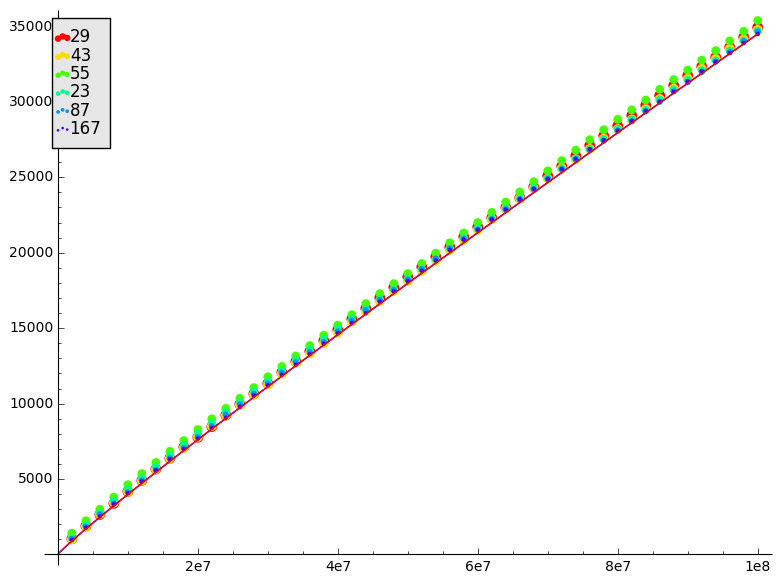}
		\caption{$\ell^k=167$}
	\end{subfigure}
\caption{Plots of $\#\{p<x \textnormal{ prime} \mid Z_p \equiv 0 \mod \ell^k\}$ for large image (line) and actual value (dots) for all eigenforms. If $\ell = 2$ no function is plotted for large image.}
\label{fig:mod}
\end{figure}

\subsection{Main Result} 

Next we test Theorem \ref{theorem:MainResult} by  comparing the behaviour of $\#\{p<x \textnormal{ prime} \mid a_p \in \Q \}$ with $ c_N\sqrt x /\log x$  where $ c_N$ is the constant predicted by Theorem \ref{theorem:MainResult}. 
Recall that  according to our main theorem under Assumptions \ref{ass:PmAss} and \ref{ass:ind} and the generalized Sato-Tate conjecture 
$$c_N  = \frac{16\sqrt D}{3\pi ^2} \widehat F.$$
Since $\widehat F$ is a limit by divisibility we approximate it numerically. In order to do so we use the following assumption.
\begin{assumption}\label{ass:IndPrimes}
Let  $m$ and $m'$  be co-prime integers. Then
 $$P_m(x) \cdot P_{m'}(x) \sim P_{m\cdot m'}(x).$$ 
\end{assumption}
If $\widehat \rho \ $ is an independent system of representations in the sense of \cite[Section 3]{ser10}, the assumption holds. However $\widehat \rho\ $ is in general not an independent system and the assumption is a much weaker claim. Moreover this assumption is only needed to get a numerical result and our main theorem holds even if this assumption is false. All computations support the assumption.
\begin{figure}[h!]
	\begin{subfigure}{0.329\textwidth}
		\includegraphics[width =  \linewidth ]{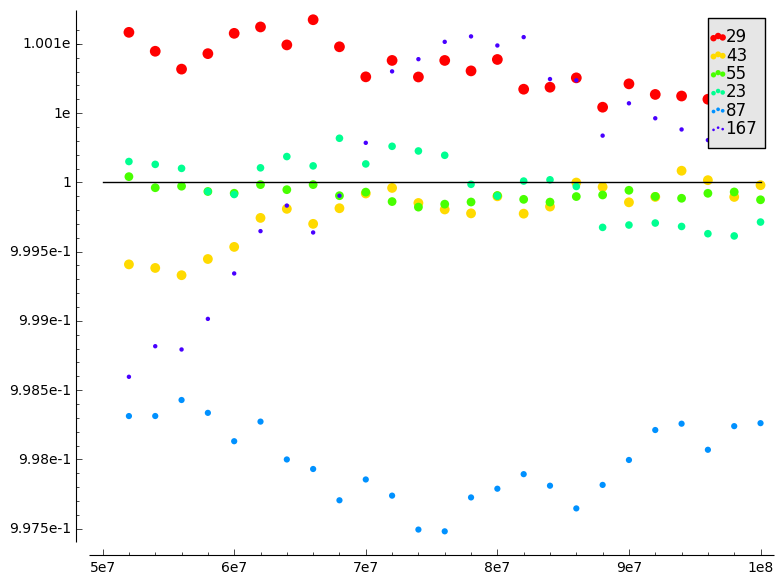}
		\caption{$m\cdot m' = 8\cdot 3 $}
	\end{subfigure}
	\begin{subfigure}{0.329\textwidth}
		\includegraphics[width =  \linewidth ]{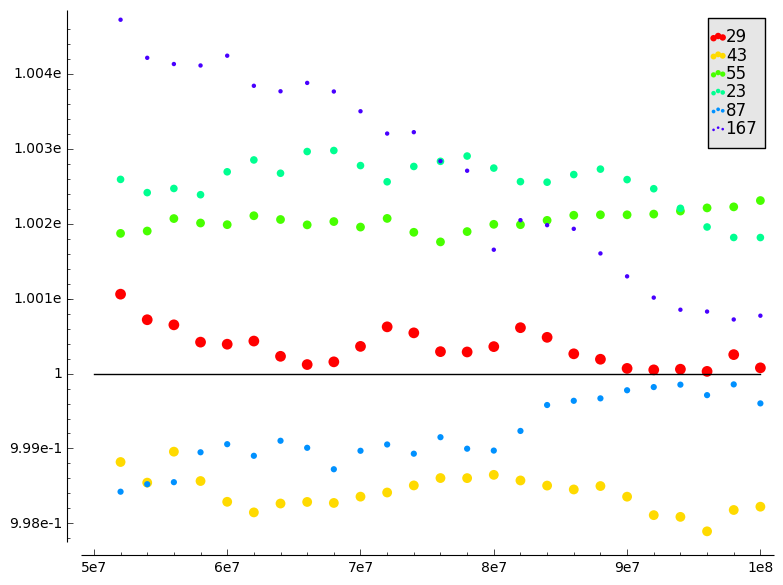}
		\caption{$m\cdot m' = 8\cdot 5 $}
	\end{subfigure}
	\begin{subfigure}{0.329\textwidth}
		\includegraphics[width =  \linewidth ]{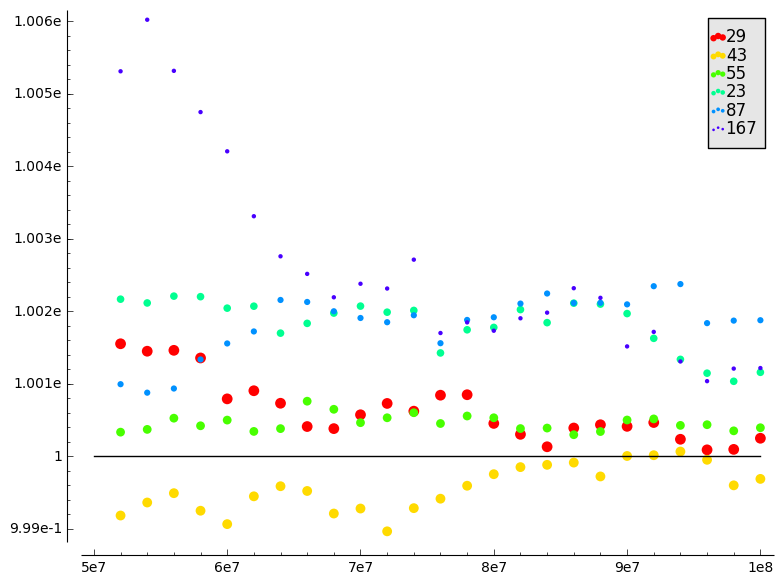}
		\caption{$m\cdot m' = 8\cdot 7 $}
	\end{subfigure}
	\begin{subfigure}{0.329\textwidth}
		\includegraphics[width =  \linewidth ]{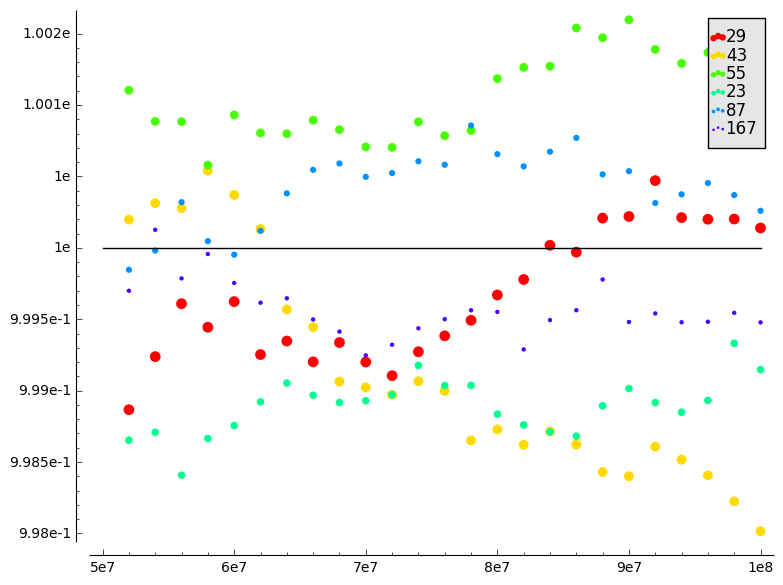}
		\caption{$m\cdot m' = 3\cdot 5 $}
	\end{subfigure}
	\begin{subfigure}{0.329\textwidth}
		\includegraphics[width =  \linewidth ]{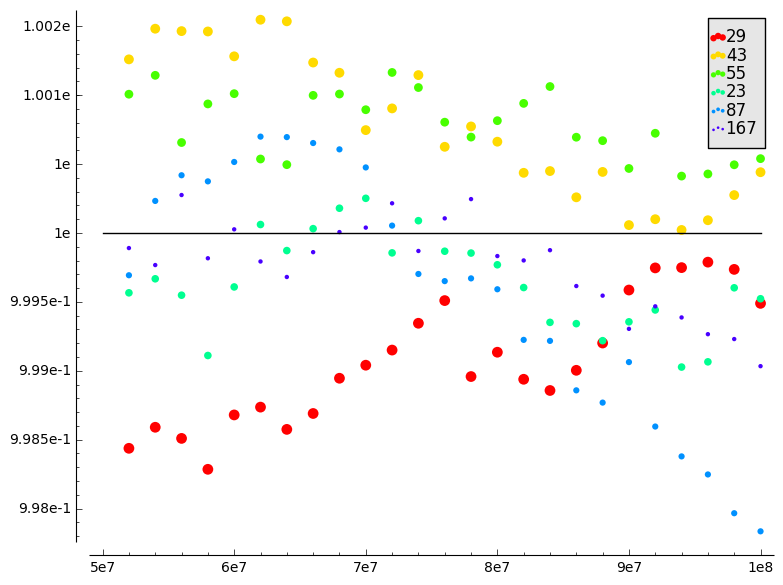}
		\caption{$m\cdot m' = 3\cdot 7 $}
	\end{subfigure}
	\begin{subfigure}{0.329\textwidth}
		\includegraphics[width =  \linewidth ]{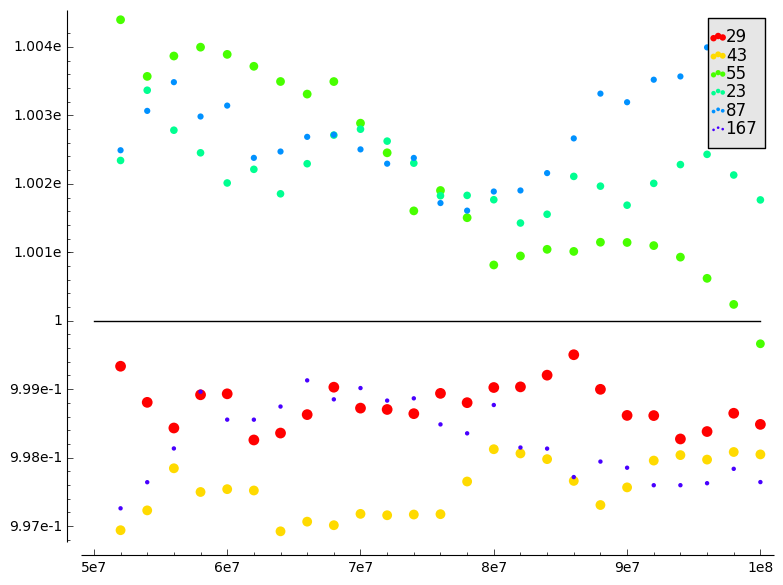}
		\caption{$m\cdot m' = 5\cdot 7 $}
	\end{subfigure}
\caption{Plots of $P_m(x)\cdot P_{m'}(x)/P_{m\cdot m'}(x) $ .}
\label{fig:MainResult}
\end{figure}

Under  Assumption \ref{ass:IndPrimes} we can compute an approximation of $\widehat F$ by taking the product over all primes
$$\widehat F = \prod _\ell \widehat F_\ell.$$
For odd unramified primes with large image the factor $\widehat F_\ell$ is given by Lemma \ref{lemma:FlkFinte}. Let $\{\ell_1,...,\ell_t\}$ be the set of primes that are even, ramified or possibly exceptional.  For every prime $ \ell_i$ we apply Lemma \ref{lemma:chebdens} for $\ell_i^{k_i}$ with $k_i$  the largest integer such that $\ell_i^{k_i}$ is less than $\sqrt {10^8} /20$, i.e.,$$k_i = \lfloor \log_{\ell_i}( \sqrt {10^8}/20) \rfloor.$$  So we use the following approximation for  $c_N$
$$\widehat c_N = \frac{16\sqrt D}{3\pi ^2}\prod_{i=1}^t \ell^{k_i}P_{\ell_i^{k_i}}(10^8) \prod _{\substack{ \ell \text { unramified} \\ \textnormal{with large image} }} \widehat F_\ell. $$
For every eigenform $f_N$ we plot $\widehat c_N \sqrt c / \log x$, $\widehat c_N \pi(x)/\sqrt x$ and $\#\{p<x \textnormal{ prime} \mid a_p \in \Q \}$ (see figure \ref{fig:MainResult}).

Comparing the values of $\widehat c_N$ to the  previously found $\widetilde c_N$ by least square fitting yields $1.025 <\widetilde c_N/ \widehat c_N <1.149 $  (Table \ref{table:FN}). This error is to be expected for this small a  bound on the primes. For example in the proof of Corollary \ref{cor:PmCor} we use $\frac{\sqrt x}{\log x}$ to approximate  $\sum_{p=2}^{x} \frac{1}{2\sqrt p} $. For  $x=10^8$ this estimate yields a similar error
$$\frac{\log {10^8}}{\sqrt {10^8}} \cdot \sum_{p=2}^{10^8} \frac{1}{2\sqrt p} =1.146\cdots.$$
\begin{figure}[h!]
	\begin{subfigure}{0.329\textwidth}
		\includegraphics[width =  \linewidth ]{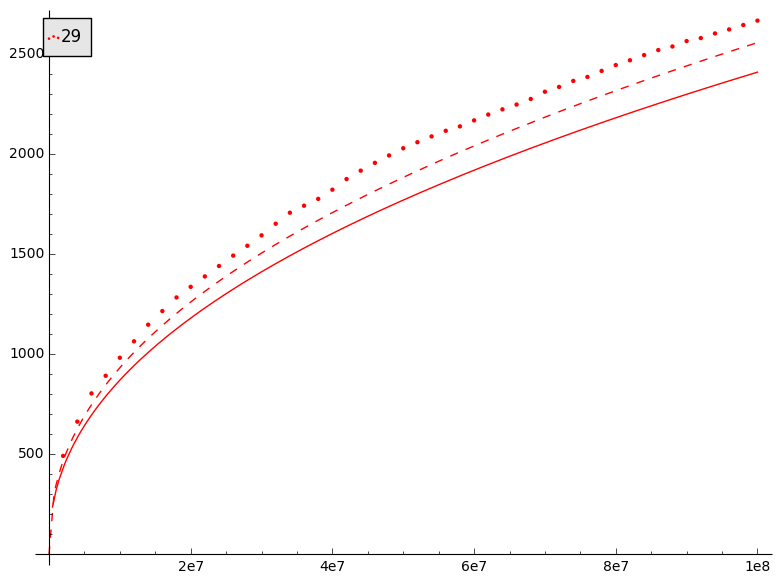}
		\caption{$N=29$}
	\end{subfigure}
	\begin{subfigure}{0.329\textwidth}
		\includegraphics[width =  \linewidth ]{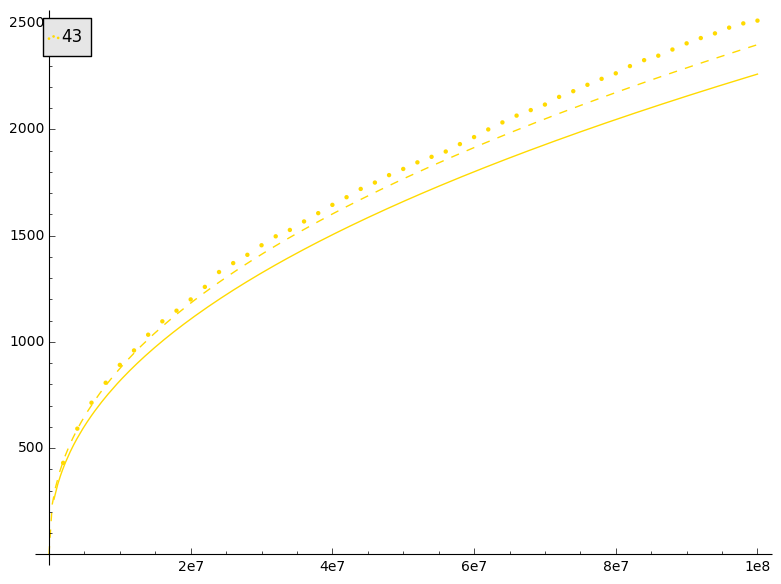}
		\caption{$N=43$}
	\end{subfigure}
	\begin{subfigure}{0.329\textwidth}
		\includegraphics[width =  \linewidth ]{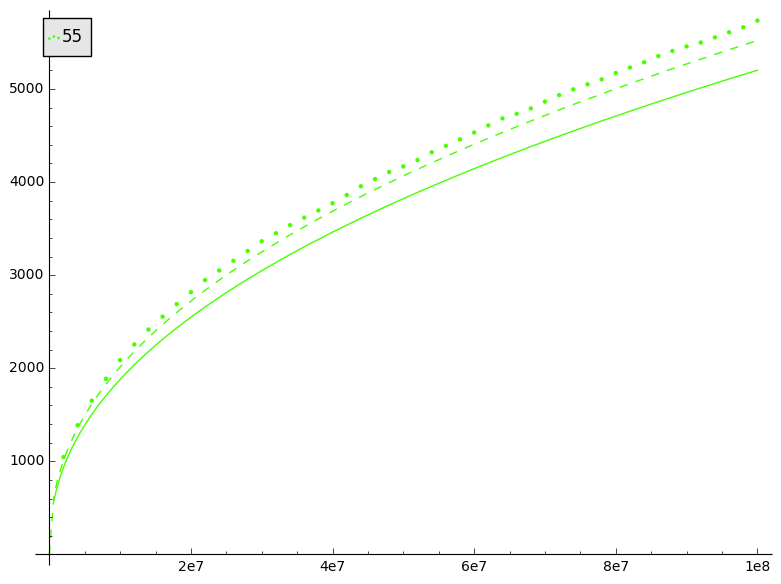}
		\caption{$N=55$}
	\end{subfigure}
	\begin{subfigure}{0.329\textwidth}
		\includegraphics[width =  \linewidth ]{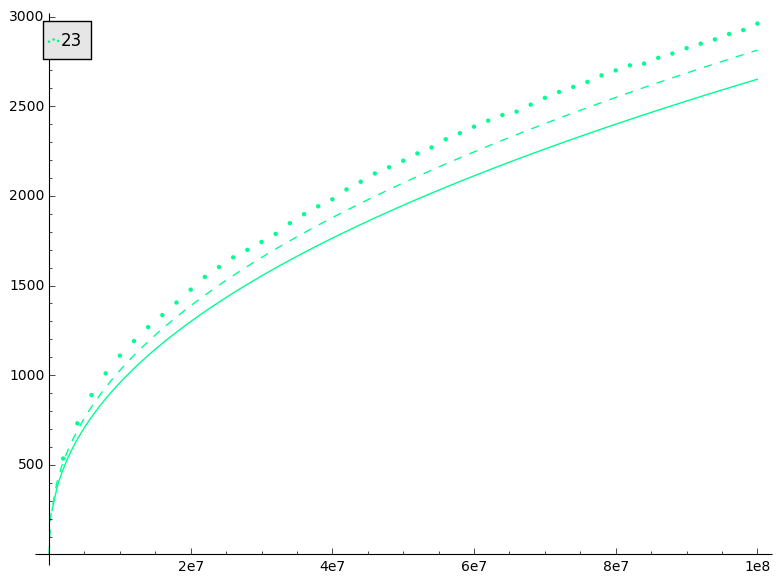}
		\caption{$N=23$}
	\end{subfigure}
	\begin{subfigure}{0.329\textwidth}
		\includegraphics[width =  \linewidth ]{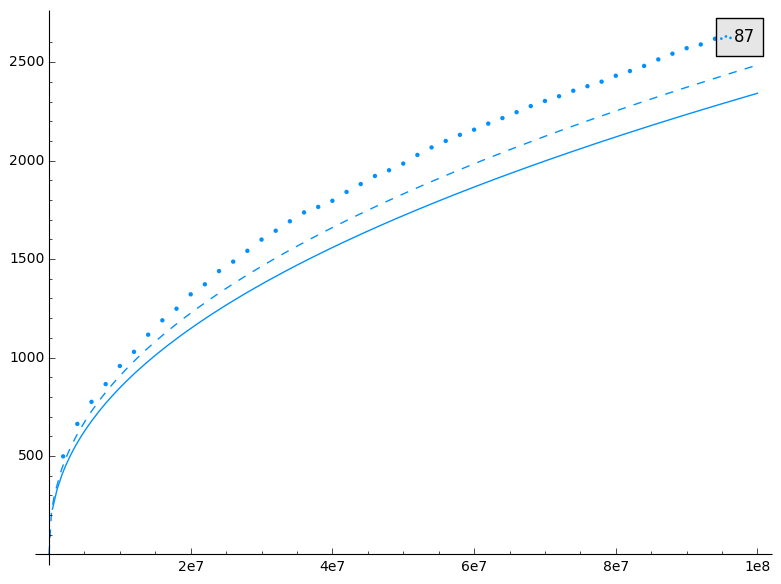}
		\caption{$N=87$}
	\end{subfigure}
	\begin{subfigure}{0.329\textwidth}
		\includegraphics[width =  \linewidth ]{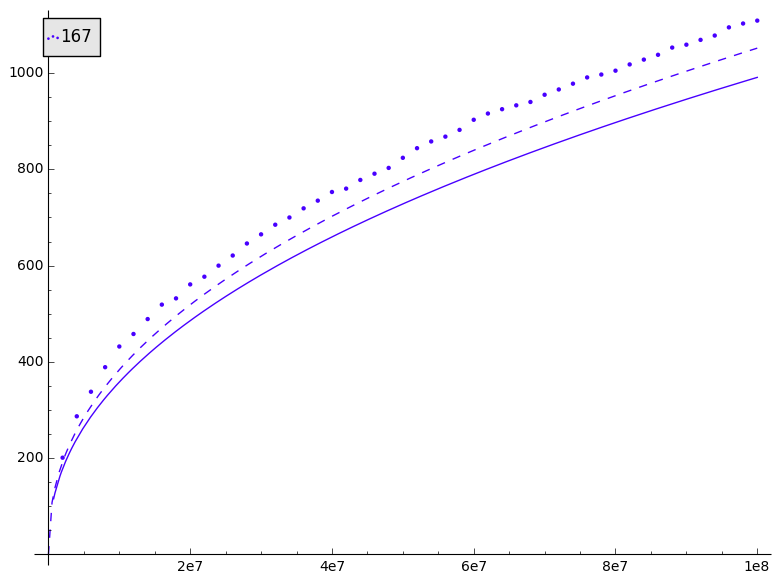}
		\caption{$N=167$}
	\end{subfigure}
\caption{Plots of $\#\{p<x \textnormal{ prime} \mid a_p \in \Q \}$ (dots), $\widehat c_N \frac{\sqrt x}{\log x}$ (full line) and $\widehat c_N\frac{\pi(x)}{\sqrt x}$ (dashed line) for all eigenforms $f_N$ with $\widehat c_N$ based on Theorem \ref{theorem:MainResult} .}
\label{fig:MainResult}
\end{figure}

\begin{table}[h]
\centering
\caption{For each eigenform $f_N$ the table contains the level $N$, the constant $\widetilde c_N$ obtained by least square fitting, the constant $\widehat c_N$ according to  Theorem \ref{theorem:MainResult}, the error $\widetilde c_N/\widehat c_N$ and the possibly exceptional primes according to Corollary \ref{cor:posExcP} respectively. The confirmed exceptional primes are marked in bold.}
\begin{tabular}{rrrrl} 
N & $\widetilde c_{N}\ $ & $\widehat c_N\ $ & $\widetilde c_N/\widehat c_N$&  Pos. exc. primes\\
\hline
29 &4.990& 4.517&1.104 & 2, 3, 5, \textbf{7}, 29\\
43 & 4.588&4.204&1.109 &2, \textbf{3}, 5, \textbf{7}, 11, 43\\
55 & 10.515&9.958&1.056 &2, \textbf{3}, 5 ,11\\
23 & 5.490&4.982&1.102 &2, 3, 5, \textbf{11}, 23\\
87 & 4.972&4.413&1.127 &2, 3, \textbf{5}, 7, 29\\
167 & 2.066&1.833&1.127 &2, 3, 5, 7, 83, 167
\end{tabular}

\label{table:FN} 
\end{table}
\pagebreak

\subsection{Final Assumption}
The final assumptions we check  are  Assumptions \ref{ass:ind} and  \ref{ass:ind1}. Recall that Assumption \ref{ass:ind} states that for every eigenform $f_N$ and every $\varepsilon>0$ there exists an $m_0$ such that for all $m$ with $m_0| m$ there exists an $x_0$ such that for all
 $x>x_0$ $$\left| \frac{P^{m}(x)\cdot P_{m}(x)}{P(x)} -1\right| <\varepsilon.$$
Assumption \ref{ass:ind1} is a much weaker claim and states that the limit 
$$ \lim_{x\rightarrow \infty} \frac{P^{m}(x)\cdot P_{m}(x)}{P(x)}  $$ 
exists for a given  positive $m$.
For all eigenforms we can find various $m$ such that  
$$\left| \frac{P^{m}(x)\cdot P_{m}(x)}{P(x)} -1\right| <0.2$$
for all $x$ larger than $5\cdot 10^7$. For every eigenform we choose different values for $m$ and plot   $\frac{P^{m}(x)\cdot P_{m}(x)}{P(x)}$ and  the constant functions  $1$ and $\frac{P^{m_N}(10^8)\cdot P_{m_N}(10^8)}{P(10^8)}$ (Fig. \ref{fig:Ind}). Where $m_N$ is the largest positive integer used for every eigenform $f_N$. The values of $m$ are chosen so that they increase by divisibility and so that the confirmed exceptional primes divide $m$.

Additionally figure \ref{fig:Ind} provides numerical evidence for Assumption \ref{ass:ind1} which implies the existence of  the double limit of $P_m(x)\cdot P^m(x)/P(x)$ by Corollary \ref{cor:MainResult}.  However one could argue that the figure  suggests that the double limit does not converge to $1$. Let us denote for every $N$
$$ \alpha_N = \limd m \lim_{x \rightarrow \infty} \frac{P^m (x)\cdot P_m(x)}{P(x)}$$ as in Corollary \ref{cor:MainResult}.  Then the corollary states that 
$$\#\{p  < x \textnormal{ prime } \mid a_p(f_N) \in \Q\} \sim \frac{1}{\alpha_N} \frac{16 \sqrt D \widehat F} {3\pi ^{2} }\frac{\sqrt x}{\log x}.$$

We have a convincing estimate $\widehat c_N$ for $c_N$. Moreover  $\alpha_{m_N} = {P^{m_N}(10^8)\cdot P_{m_N}(10^8)}/{P(10^8)}$ is the best approximation of $\alpha_N$ available. So we can  check this   last statement by plotting both functions  (Fig. \ref{fig:MainResultAPi}). In this figure  $1/\alpha_{m_N} \widehat c_N\frac{\pi(x)}{\sqrt x}$  clearly yields an overestimate when we in fact expect a slight underestimate. This is an indication that, although the convergence might be slow, the double limit equals $1$.
\begin{figure}[h!]
	\begin{subfigure}{0.329\textwidth}
		\includegraphics[width =  \linewidth ]{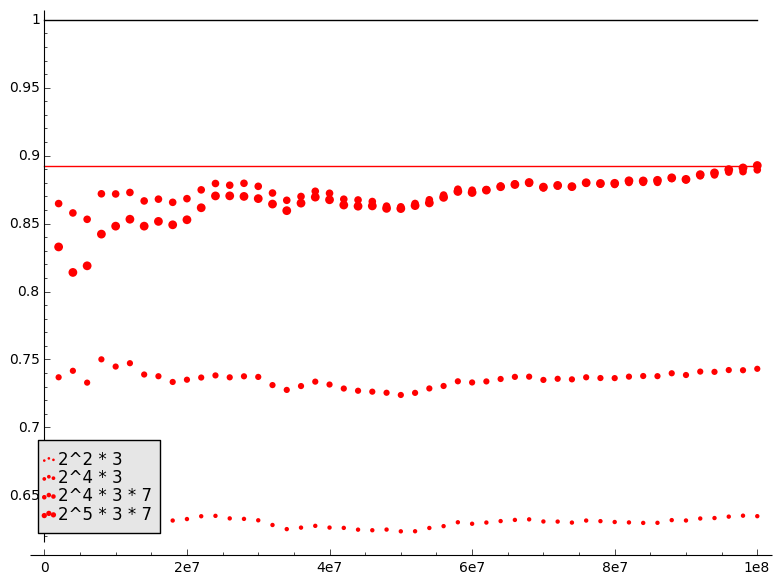}
		\caption{$N=29$}
	\end{subfigure}
	\begin{subfigure}{0.329\textwidth}
		\includegraphics[width =  \linewidth ]{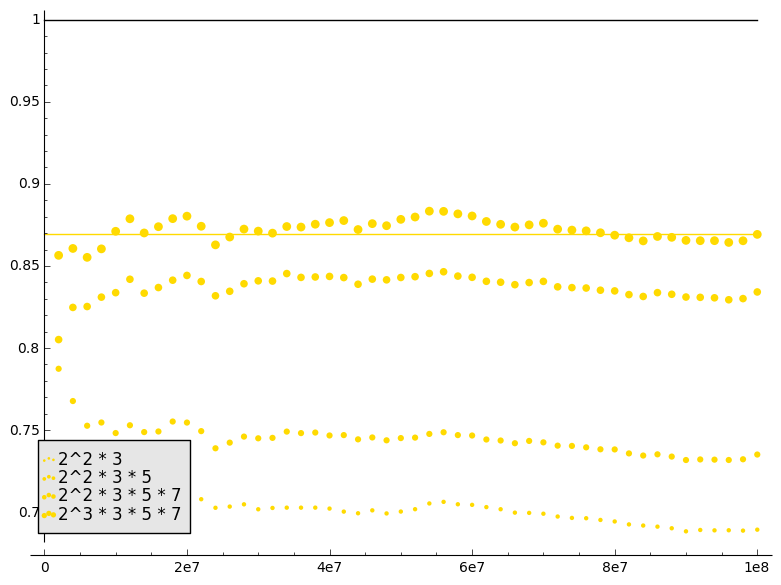}
		\caption{$N=43$}
	\end{subfigure}
	\begin{subfigure}{0.329\textwidth}
		\includegraphics[width =  \linewidth ]{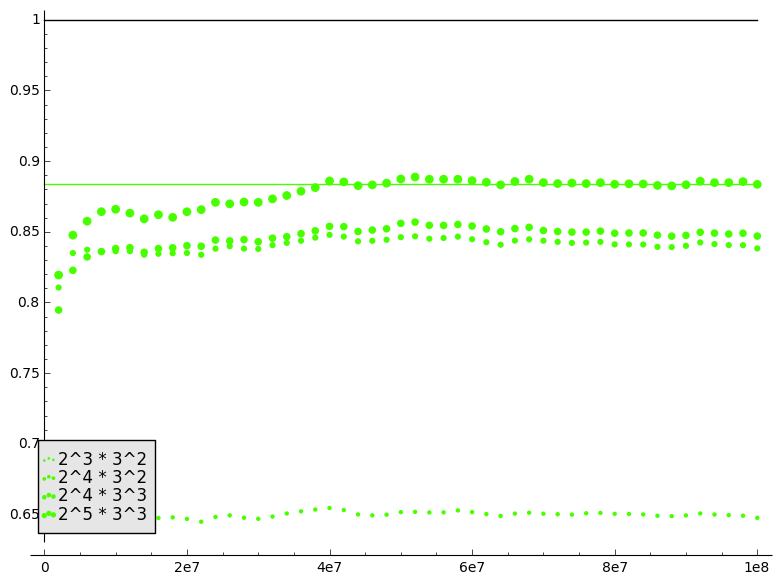}
		\caption{$N=55$}
	\end{subfigure}
	\begin{subfigure}{0.329\textwidth}
		\includegraphics[width =  \linewidth ]{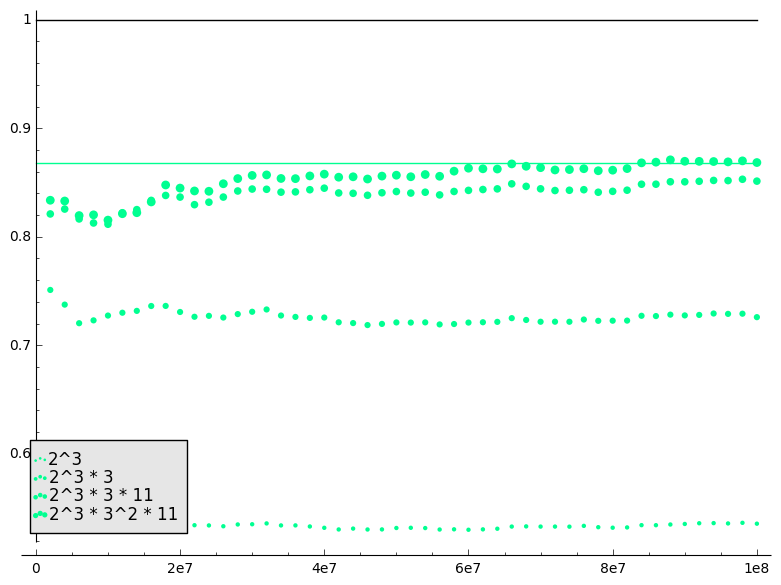}
		\caption{$N=23$}
	\end{subfigure}
	\begin{subfigure}{0.329\textwidth}
		\includegraphics[width =  \linewidth ]{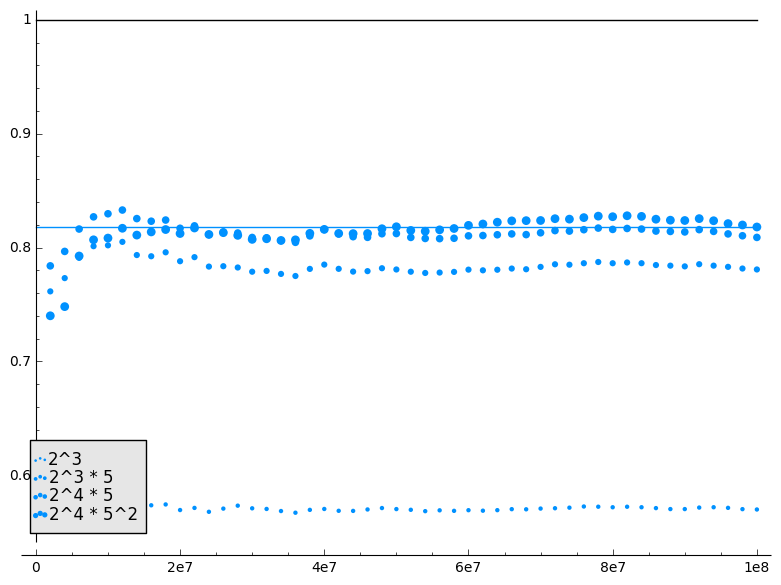}
		\caption{$N=87$}
	\end{subfigure}
	\begin{subfigure}{0.329\textwidth}
		\includegraphics[width =  \linewidth ]{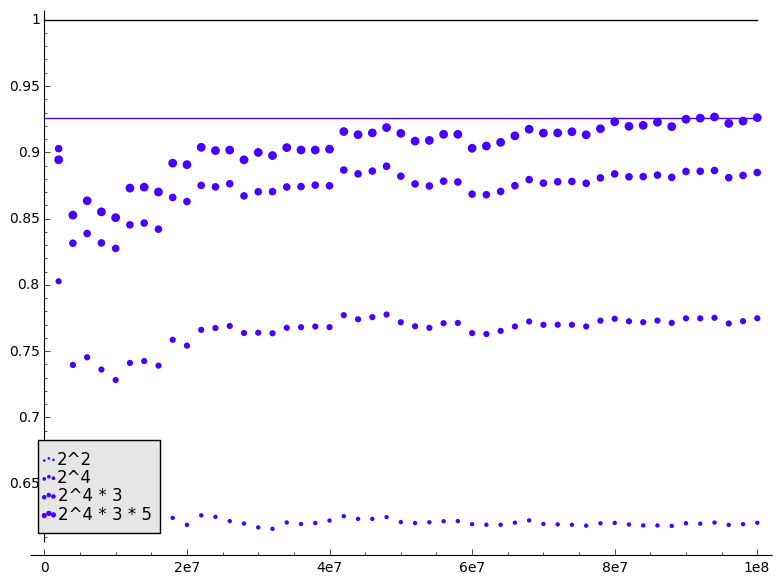}
		\caption{$N=167$}
	\end{subfigure}
\caption{Plots of ${P^{m}(x)\cdot P_{m}(x)}/{P(x)}$ (dots) and the constant functions $1$ (black line) and ${P^{m_N}(10^8)\cdot P_{m_N}(10^8)}/{P(10^8)}$ (coloured line) for each eigenform $f_N$ and various divisors $m$ of $m_N$ for $x$ up to $10^8$.}
\label{fig:Ind}
\end{figure}

\begin{figure}[h!]
	\begin{subfigure}{0.329\textwidth}
		\includegraphics[width =  \linewidth ]{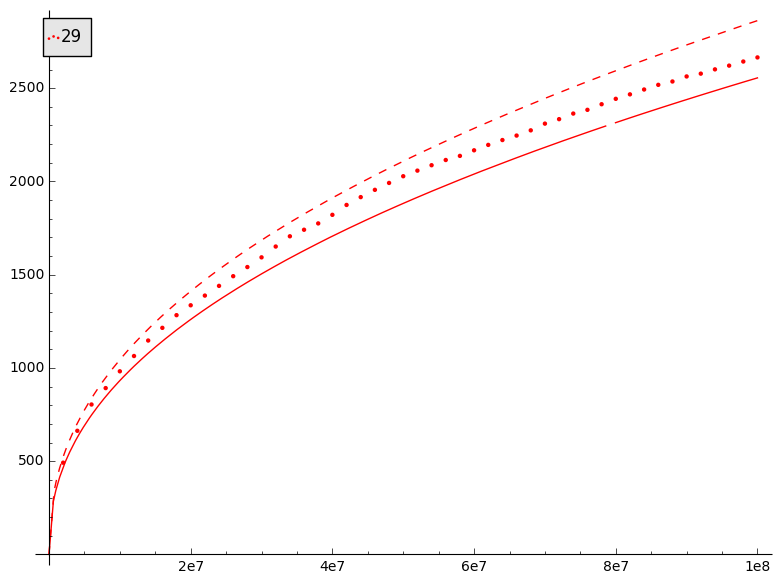}
		\caption{$N=29$}
	\end{subfigure}
	\begin{subfigure}{0.329\textwidth}
		\includegraphics[width =  \linewidth ]{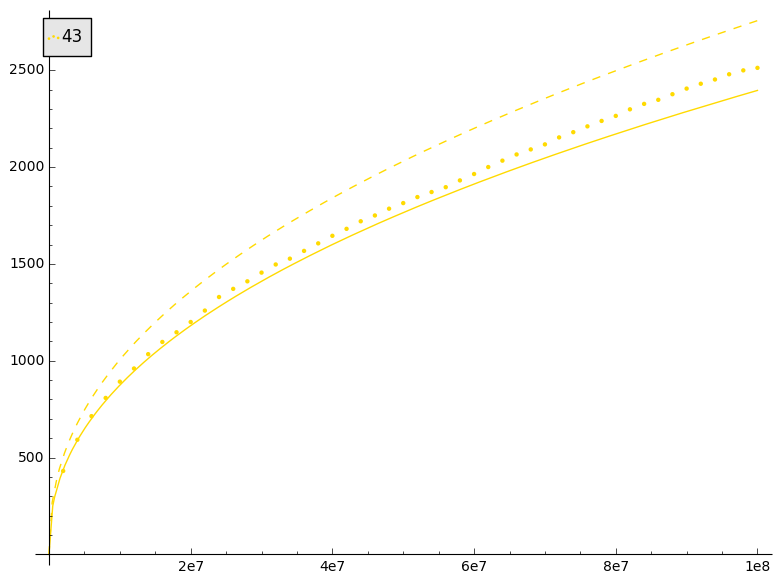}
		\caption{$N=43$}
	\end{subfigure}
	\begin{subfigure}{0.329\textwidth}
		\includegraphics[width =  \linewidth ]{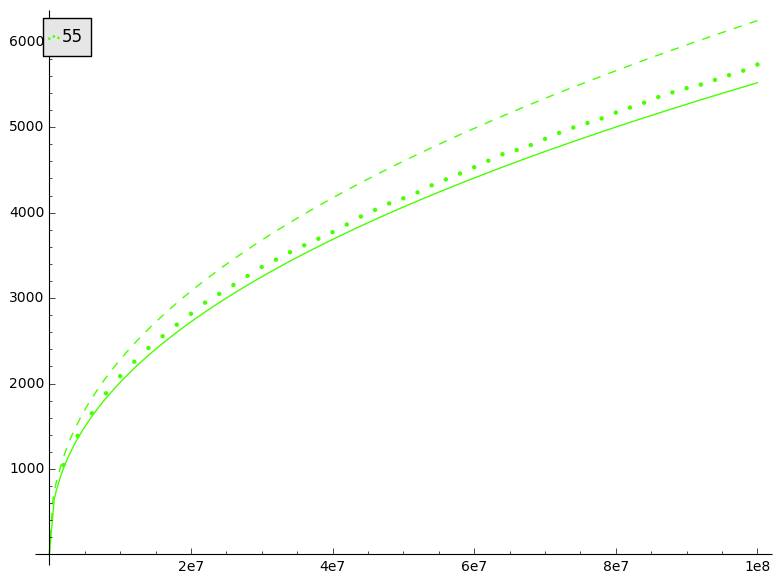}
		\caption{$N=55$}
	\end{subfigure}
	\begin{subfigure}{0.329\textwidth}
		\includegraphics[width =  \linewidth ]{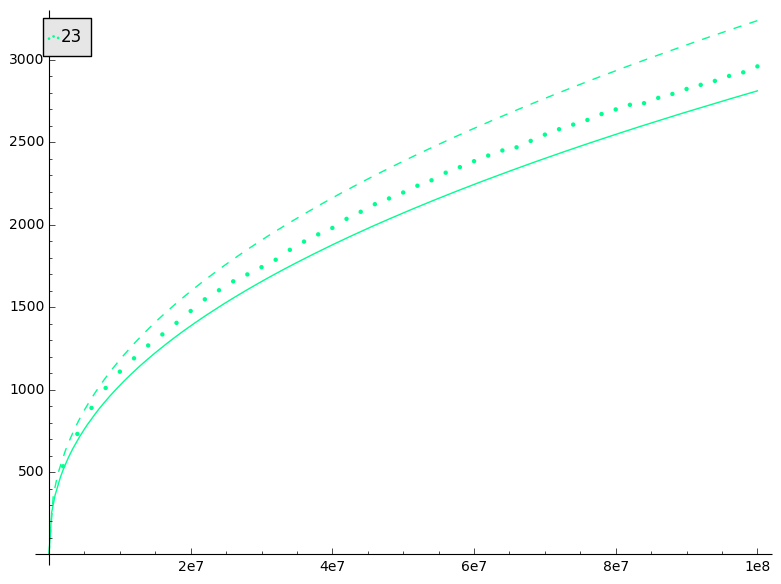}
		\caption{$N=23$}
	\end{subfigure}
	\begin{subfigure}{0.329\textwidth}
		\includegraphics[width =  \linewidth ]{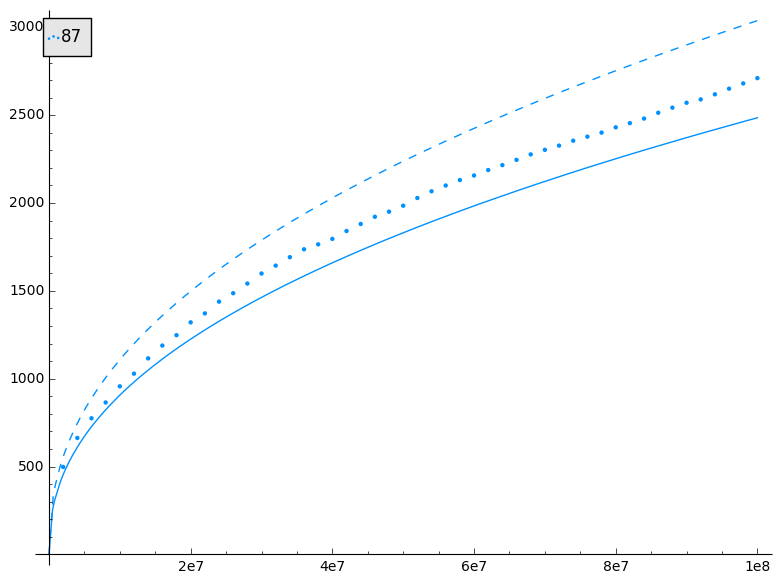}
		\caption{$N=87$}
	\end{subfigure}
	\begin{subfigure}{0.329\textwidth}
		\includegraphics[width =  \linewidth ]{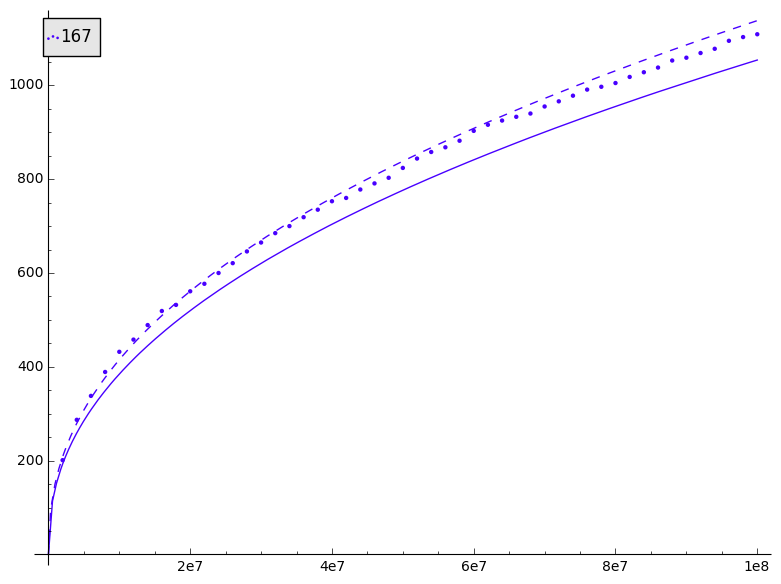}
		\caption{$N=167$}
	\end{subfigure}
\caption{Plots of $\#\{p<x \textnormal{ prime} \mid a_p \in \Q \}$ (dots), $\widehat c_N \frac{\pi (x)}{\sqrt x}$ (full line) and $ 1/\alpha_{m_N}\, \widehat c_N\frac{\pi(x)}{\sqrt x}$ (dashed line)  for all eigenforms $f_N$ with $\widehat c_N$ based on Theorem \ref{theorem:MainResult} .}
\label{fig:MainResultAPi}
\end{figure}
\pagebreak

\clearpage
\appendix
\section{Proof of Proposition \ref{prop:GLKTInert}}\label{app:GLKTInert}
In this appendix we give the proof of the following proposition.
\begin{propn}[\ref{prop:GLKTInert}]
Let $\ell$ be an odd prime and $k$ a positive integer. Then
$$\#\A_{\ell^k,I}^t =\frac{(\ell-1)}{(\ell+1)}\ell^{6k-2}\left(\ell^2+\ell +1-\ell^{-2k} \right).$$ 
\end{propn}
Before we give the proof we need a lemma. Let $$\v : \Z [ \alpha] \rightarrow \Z $$ be the $\ell$-adic valuation. By abuse of notation we will also use $\v$ to denote the induced  valuation on $\Z[\alpha]/\ell^k\Z[\alpha]$.

\begin{lemma}\label{lemma:PairsAD}
	Consider for any  $k\ge r \ge 1$  the following two conditions on pairs \ctext{$(a,d) \in (\Z[\alpha]/\ell^k\Z[\alpha])^2$}
	\begin{numcases} {}
		 a\cdot d \in \Z/\ell^k\Z^\times + \alpha \ell^r \Z/\ell^k\Z \label{eq:cond1} \\
		 a+d \in \Z/\ell^k\Z. \label{eq:cond2}
	\end{numcases}
	\begin{enumerate}
		\item Let  $b\in \Z[\alpha]/\ell^k\Z[\alpha]$ with $\textnormal{min}\{\v (b),k\} = r$. Then there exists an element $c\in \Z[\alpha]/\ell^k\Z[\alpha]$ such that $\begin{pmatrix} a & b \\ c & d\end{pmatrix} \in \A_{\ell^k,I}^t$ if and only if both \eqref{eq:cond1} and \eqref{eq:cond2} hold. Moreover there exist precisely  $\ell^{k+r}$ distinct such $c$.
		\item There are $$\ell^{3k-r-2}(\ell-1)(r\ell-r+\ell)$$ pairs $(a,d) \in (\Z[\alpha]/\ell^k\Z[\alpha])^2$ satisfying \eqref{eq:cond1} and \eqref{eq:cond2}.
	\end{enumerate}
\end{lemma}

\begin{proof}
\begin{enumerate}
	\item The 'only if' statement is immediate. Conversely suppose that $a$ and $d$ are elements of $\Z[\alpha]/\ell^k \Z[\alpha]$ satisfying \eqref{eq:cond1} and \eqref{eq:cond2}. Let $b \in \Z[\alpha]/\ell^k \Z[\alpha]$ and $k\ge r \ge 1$ with $ r =\min {\v(b)} k$. For every $c$ in $\Z[\alpha]/\ell^k\Z[\alpha]$ denote
	$$\sigma_c :=\begin{pmatrix} a & b \\ c & d \end{pmatrix}.$$
	Condition \eqref{eq:cond2} is equivalent with $\tr \sigma_c \in \Z/\ell^k\Z$. So it suffices to show that there are $\ell^{k+r}$ distinct $c$ such that $\det \sigma_c \in \Z/\ell^k\Z^\times$.

	Let $e_1$, $e_2$, $b_1$, $b_2$, $c_1$ and $c_2$ be elements of $\Z/\ell^k\Z$ such that
	\begin{align*}
		a\cdot d &= e_1 + \ell^re_2\alpha, \\
		b &=\ell^r(b_1 + b_2\alpha ), \\
		c &=c_1 + c_2\alpha . \\
	\end{align*}
	Then $\det \sigma_c \in \Z/\ell^k\Z^\times$ if and only if
	\begin{align*}
		e_1 - \ell^r(b_1c_1 + b_2 c_2 \alpha^2) & \in \Z/\ell^k\Z^\times \textnormal{ and} \\
		\ell^r(e_2-b_2c_1 - b_1c_2) & = 0.
	\end{align*}
	Note that \eqref{eq:cond1} implies that $e_1$ is invertible. So the first condition is satisfied.  Moreover  $\v(b) \ge r$ so either $b_1$ or $b_2$ is a unit. Without loss of generality suppose that $b_1$ is a unit. Then $\det \sigma_c \in \Z/\ell^k\Z$ if and only if 
	$$\ctext{c_2 \in b_1^{-1}(e_2-b_2c_1) + \ell^{k-r}\Z/\ell^k\Z.}$$
	In particular there exist $\ell^{k+r}$ distinct such $c$ such that $\sigma_c \in \A_{\ell^k,I}^t$.
	\item Let $a_1$, $a_2$, $d_1$ and $d_2$ be elements of $\Z/\ell^k\Z$ such that 	$a= a_1 + a_2 \alpha$  and $d = d_1 + d_2 \alpha$.
	Then conditions \eqref{eq:cond1} and \eqref{eq:cond2} hold if and only if 
	\begin{align*}
		\begin{cases}
		a_1d_1+a_2d_2 \alpha^2 \in \Z/\ell^k\Z^\times \\
		a_1d_2+a_2d_1 \in \ell^r\Z/\ell^k\Z \\
		a_2 +d_2 = 0.
		\end{cases}
	\end{align*}
	{Using $d_2 =- a_2$ in the first expressions we obtain}
	\begin{numcases}{}
		a_1d_1 - a_2 ^2\alpha^2 \in \Z/\ell^k\Z^\times \label{eq:cond3}\\
		a_2(a_1- d_1) \in \ell^r \Z/\ell^k\Z. \label{eq:cond4}
	\end{numcases}
	We distinguish three cases depending on the valuation of $a_2$.
	\begin{enumerate}
		\item If $a_2= 0$, then \eqref{eq:cond4} is satisfied and 
		\begin{align*}
			\eqref{eq:cond3} &\Leftrightarrow a_1d_1 \in \Z/\ell^k\Z^\times.
		\end{align*}
		So there are $\ell^{2k-2}(\ell-1)^2$ such pairs $(a,d)$.
		\item If $r \le \v(a_2)  = s <k$, then $a_2\in \ell^r\Z/\ell^k \Z$ so \eqref{eq:cond4} is satisfied and
		\begin{align*}
			\eqref{eq:cond3} & \Leftrightarrow a_1d_1 \in \Z/\ell^k\Z^\times.
		\end{align*}
		So for every $r\le s < k$ there are $\ell^{3k-3-s}(\ell-1)^3$ pairs $(a,d)$ with $\v (a_2) = s$ that satisfy \eqref{eq:cond1} and \eqref{eq:cond2}.
		\item If $0<\v(a_2) = s<r$ then
			\begin{align*} & \ctext{\begin{cases}
					\eqref{eq:cond3} \Leftrightarrow a_1d_1 \in \Z/\ell^k\Z^\times\\
					\eqref{eq:cond4} \Leftrightarrow a_1-d_1 \in \ell^{r-s}(\Z/\ell^k\Z^\times).
			\end{cases}}
			\intertext{These conditions are equivalent to} 
			&  \begin{cases}
					a_1 \in \Z/\ell^k\Z^\times \\
					d_1 \in a_1 + \ell^{r-s} \Z/\ell^k\Z.
			\end{cases} \end{align*}	
		So for every $0<s<r$ there are $\ell^{3k-2-r}(\ell-1)^2$ such pairs $(a,d)$ with $\v(a_2) = s$.		
		\item Finally suppose that $a_2$ is invertible. If $a_1 \in \ell\Z/\ell^k \Z$, then \eqref{eq:cond4} implies \eqref{eq:cond3}. Hence the remaining condition is
		\begin{align*}
		d_1 \in a_1 +\ell^r \Z/\ell^k \Z.
		\end{align*}
		So there are $\ell^{3k-2-r}(\ell-1)$ pairs $(a,d)$ with $a_2$ a unit and $a_1$ not invertible.
		
		If $a_1$ is an invertible element we have
		\begin{align*} & \begin{cases}
					\eqref{eq:cond3} \Leftrightarrow d_1 \not \in a_1^{-1}a_2^2 \alpha^2+  \ell\Z/\ell^k\Z\\
					\eqref{eq:cond4} \Leftrightarrow d_1 \in a_1 + \ell^{r}\Z/\ell^k\Z.
			\end{cases} 
			\intertext{Note that the intersection of $a_1^{-1}a_2^2 \alpha^2+  \ell\Z/\ell^k\Z$ and $a_1 + \ell^{r}\Z/\ell^k\Z $ is empty. Indeed, otherwise $a_1  \equiv a_1^{-1}a_2^2 \alpha^2 \mod \ell$ or $(a_1a_2^{-1})^2 \equiv \alpha^2 \mod \ell$. Which contradicts the fact that $\alpha^2$ is a quadratic non-residue modulo $\ell$. Hence}
			& \eqref{eq:cond4} \Leftrightarrow d_1 \in a_1 + \ell^r\Z\ell^k \Z \Rightarrow \eqref{eq:cond3}.
		\end{align*}
		In particular there are $\ell^{3k-2-r}(\ell- 1)^2$ pairs $(a,d)$ with both $a_1$ and $a_2$ invertible elements. 
		
		Adding the  two cases, $a_1$ being a unit and $a_1$ not being a unit, yields
		\begin{align*}
			\ell^{3k-2-r}(\ell- 1)^2 + \ell^{3k-2-r}(\ell-1) & = \ell^{3k-2-r}(\ell- 1)(\ell- 1 +1) \\
			& =\ell^{3k-1-r}(\ell- 1)
		\end{align*}
		distinct pairs $(a,d)$ satisfying \eqref{eq:cond1} and \eqref{eq:cond2} with $a_1$ invertible.
	\end{enumerate}
	Finally we sum over all cases
	\begin{align*}
			 & \ell^{2k-2}(\ell-1)^{2} + \sum_{s=r}^{k-1} \ell^{3k-3-s}(\ell -1)^{3} + \sum_{ s=1}^{r-1} \ell^{3k-2-r}(\ell -1)^{2} + \ell^{3k-1-r}(\ell-1) \\
			 & = \ell^{2k-2}(\ell-1)\Big(\ell-1 + (\ell-1)\sum_{s=r}^{k-1}\ell^{k-1-s}(\ell-1) + \sum_{s=1}^{r-1}\ell^{k-r}(\ell-1) + \ell^{k-r+1}  \Big). \\
			 \intertext{Note that the first summation is telescopic so the total sum yields }
			 &   \ell^{2k-2}(\ell-1)\Big( \ell -1 + (\ell-1)(\ell^{k-r}-1) + (r-1)(\ell-1)\ell^{k-r}+ \ell^{k-r+1} \Big){}\\
			 & = \ell^{2k-2}(\ell-1)\Big(r(\ell-1) \ell^{k-r} + \ell^{k-r+1}  \Big){} \\
			 &= \ell^{3k-2-r}(\ell-1)\Big(r(\ell-1) + \ell  \Big). \qedhere
	\end{align*}	
\end{enumerate}

\end{proof}
\begin{proof}[Proof of Proposition \ref{prop:GLKTInert} ] $\\$
	Let $\sigma = \begin{pmatrix} a& b \\ c & d \end{pmatrix} \in GL_2(\Z[\alpha]/\ell^k\Z[\alpha])$. Then $\sigma \in \A_{\ell^k,I}^t$ if and only if
	\begin{numcases}{}
		ad-bc \in \Z/\ell^k \Z^\times \label{eq:cond5} \\
		a+d \in \Z/\ell^k \Z \label{eq:cond6}.
	\end{numcases}
	We distinguish three cases depending on the valuation of $b$.
	\begin{enumerate}
		\item If $b$ is a unit then
		\begin{align*} \begin{cases} 
		 \eqref{eq:cond5} \Leftrightarrow c\in adb^{-1} + \Z/\ell^k \Z^\times \\
		 \eqref{eq:cond6} \Leftrightarrow d\in -a+\Z/\ell^k \Z.
		\end{cases} \end{align*}
		So there are $$\ell^{2k} \cdot \ell^{2k-2}(\ell^2 -1) \cdot \ell^{k-1}(\ell-1) \cdot \ell^k = \ell^{6k-3}(\ell-1)(\ell^2-1)$$
		matrices $\sigma \in \A_{\ell^k,I}^t$ with $b$ a unit.
		
		\item If $b = 0$, then 
		\begin{align*} \begin{cases} 
		 \eqref{eq:cond5} \Leftrightarrow ad\in  \Z/\ell^k \Z^\times \\
		 \eqref{eq:cond6} \Leftrightarrow d+a \in \Z/\ell^k \Z.
		\end{cases} \end{align*}
		By Lemma \ref{lemma:PairsAD} with $r=k$ there are $\ell^{2k-2}(\ell-1)(k\ell-k+\ell)$ such pairs $(a,d)$. Again by Lemma \ref{lemma:PairsAD} there are $\ell^{2k}$ elements $c$ in $\Z_{\ell^2}$ for every such pair $(a,d)$ so that  $\begin{pmatrix} a & 0 \\ c & d \end{pmatrix}$ is an element of $\A_{\ell^k,I}^t$. In particular we obtain 
		$$\ell^{4k-2}(\ell-1)(k\ell- k +\ell)$$
		elements of $\A_{\ell^k, I}^t$ with $b=0$.
		
		\item Let $0<\v(b) = r <k$. Then by  part 1 of Lemma \ref{lemma:PairsAD} there exists $\ell^{k+r}$ elements $c$ for each $b$ and  each pair \eqref{eq:cond1} and \eqref{eq:cond2} such that $\begin{pmatrix} a &  b \\ c & d \end{pmatrix} \in \A_{\ell^k,I}^t$. By part 2 of the same lemma there exist $\ell^{3k-2-r}(\ell-1)(r\ell - r + \ell)$ such pairs $(a,d)$. Moreover there are $\ell^{2k-2-2r}(\ell^2 -1)$ elements $b$ in $\Z[\alpha]/\ell^k \Z[\alpha]$ with $\v(b) = r$. So there are
		\begin{align*}
		\ell^{3k-2-r}(\ell-1)(r\ell-r+\ell)\cdot \ell^{2k-2-2r}(\ell^2-1) \cdot \ell^{k+r} \\ = \ell^{4(k-1)+2(k-r)}(\ell-1)(\ell^2-1)(r\ell - r +\ell)
		\end{align*}
		elements in $\A_{\ell^k,I}^t$ with $\v(b) = r$ for each $0<r<k$.
	\end{enumerate}
	It remains to take the sum over all three cases. Note that 
	\begin{align*}
		\#\A_{\ell^k,I}^t  &= \ell^{6k-3}(\ell-1)(\ell^2-1) + \sum_{r=1}^{k-1}\ell^{4(k-1)+2(k-r)}(\ell-1)(\ell^2-1)(r\ell - r +\ell) 
		\\&+ \ell^{4k-2}(\ell-1)(k\ell- k +\ell) \\
		=& \ell^{4(k-1)}(\ell -1) \Big(k\ell-k+\ell + \sum_{r=0}^{k}\ell^{2(k-r)}(\ell^{2}-1)(r(\ell -1)+\ell)    \Big).\\
		\intertext{We split the summation in two terms} 
		\#\A_{\ell^k,I}^t=& \ell^{4(k-1)}(\ell -1) \Big( k(\ell-1)+\ell  \\ 
		&+ (\ell-1)\sum_{r=0}^{k}r\ell^{2(k-r)}(\ell^{2}-1) +  \sum_{r=0}^{k}\ell^{2(k-r)+1}(\ell^{2}-1) \Big). \\
		\intertext{Expand the first summation and note that the  second summation is telescopic}
		\#\A_{\ell^k,I}^t=& \ell^{4(k-1)}(\ell -1) \Big( k(\ell-1)+\ell  \\ 
		&+ (\ell-1)(\ell^{2k}+\ell^{2k-2}+\cdots +\ell^{2}-k) + (\ell^{2k+3} -\ell) \Big) \\ 
		=& \ell^{4(k-1)}(\ell -1) \Big(k(\ell-1)+\ell   \\ 
		&+ (\ell-1)\frac{\ell^{2(k+1) }-1}{\ell^{2}-1} -(\ell-1)(1+k)  +\ell^{2k+3} -\ell \Big)\\
		=& \ell^{4(k-1)}(\ell -1) \Big( \frac{\ell^{2(k+1) }-1}{\ell+1} -(\ell-1) + \ell^{2k+3}\Big) \\
		=& \ell^{4(k-1)}\frac{(\ell -1)}{(\ell+1)} \Big( \ell^{2(k+1)} -1 - (\ell + 1)(\ell-1) + \ell^{2k+4} + \ell^{2k+3} \Big){} \\
		= & \ell^{4(k-1)}\frac{(\ell -1)}{(\ell+1)} \Big(\ell^{2k+4} - \ell^{2}+ \ell^{2k+3} +  \ell^{2k+2}\Big){} \\
		= & \ell^{6k-2}\frac{(\ell -1)}{(\ell+1)} \Big(\ell^{2} +\ell +1 -\ell^{-2k} \Big). \qedhere
	\end{align*} 
\end{proof}

\section{Proof of Proposition \ref{prop:GLKTSplit}}\label{app:GLKTSplit}
In this appendix we prove Proposition \ref{prop:GLKTSplit}.
\begin{propn}[\ref{prop:GLKTSplit}]
Let $\ell$ be an odd prime and $k$ a positive integer. Then
$$\# \A_{\ell^k,S}^t = \frac{(\ell-1)}{(\ell+1)} \ell^{6k-4} \left( \ell^4+\ell^3-\ell^2-2 \ell- \ell^{-2k+2} \right).$$
\end{propn}

Before we prove this proposition we need three intermediate results. 
\begin{lemma}\label{lemma:DiscModlk}
	Let $\ell$ be an odd prime and $k$ a positive integer. Let $f$ be a quadratic monic polynomial in $\Z/\ell^k\Z[X]$ and  $D$ be the discriminant of $f$. Then
	\begin{align*}
		\#\big \{r \in \Z/\ell^k\Z \mid f(r) = 0\big \} = \begin{cases}
		\ell^{\lfloor k/2 \rfloor} & \textnormal{if } D=0 \\
		2\ell^i & \textnormal{if } D= u^2 \ell^{2i} \textnormal{ with } u \in \Z/\ell^k\Z^\times \textnormal{ and } 0\leq i \leq k/2\\
		0&\textnormal{else.}  
		\end{cases}
	\end{align*}
\end{lemma}
\begin{proof}
	Let   $f = X^2 +bX + c$ be a polynomial in $\Z/\ell^k \Z[X]$. 

	If $f$ has at least one zero, say $a$, then
	\begin{align*}
		f(a) = 0 &\Leftrightarrow 4a^2 +4ba+4c = 0 \\
		&\Leftrightarrow (2a+b)^2  = b^2 -4c.
	\end{align*}
	In particular the discriminant of $f$ is a square.
	Let  $D = b^2 - 4 c$ be the discriminant of $f$. We distinguish two cases based on the valuation of $D$.
	
	If $D$ is an invertible element, the result follows from Hensel's lemma. So we may assume that the valuation of $D$ is at least one. Let $\widehat f$ be the reduction of $f$ modulo $\ell$. Then $\widehat f$ has $-b/2$ as a double zero. In particular any zero of $f$ in $\Z/\ell^k\Z$ is of the form $-b/2 + a$ with $a\in \ell\Z/\ell^k \Z$. Let us compute
	\begin{align*}
		f\left(-\frac b 2 + a\right) = 0 &\Leftrightarrow (-\frac b 2 + a)^2+ b(-\frac b 2 + a) + c = 0 \\
		& \Leftrightarrow \frac{b^2 } 4 - \frac{b^2 } 2 + c +a^2  \\
		& \Leftrightarrow -\frac D 4 +a^2= 0. \\
	\end{align*}
	If $D = 0$, the zeros of $f$ are precisely the elements of $-\frac{b}{2} + \ell^{\lceil k/2 \rceil}\Z/\ell^k\Z$. In particular $f$ has $\ell^{\lfloor k/2 \rfloor}$ zeros.
	
	Finally suppose that $D= u^2 \ell^{2i}$ with $u$ an invertible element. Then the zeros of $f$ are $$-\frac b 2 \pm \frac{u} 2 \ell^i + \ell^{k-i} \Z/\ell^k\Z.$$
	This set has cardinality $2\ell^i$.
\end{proof}

For every $0\le i \le k$ let $P_i$ be the set of monic quadratic polynomials with coefficients in $\Z/\ell^k\Z$ with invertible constant term and $\min {\v( \textnormal{Disc}\,f) } k = i $. Moreover consider the following subsets of $P_i$
\begin{align*}
P_{i,0} &=  \Big \{ f \in P_i \mid \textnormal{Disc}\,f \textnormal{ not a  quadratic residue in } \Z/\ell^k\Z\Big \},\\
P_{i,2} & =  \Big \{ f \in P_i \mid \textnormal{Disc}\,f \textnormal{  a  quadratic residue in } \Z/\ell^k\Z \Big \},
\end{align*}
where $\text{Disc}\, f$ denotes the discriminant of the quadratic polynomial $f$.
Note that  the set $P_{i,2}$ is empty  for every odd $i$  moreover so is $P_{k,0}$  since $0$ is a quadratic residue. 

\begin{lemma}\label{lemma:Plk}

Let $\ell$ be an odd prime and $k$ a positive integer. Then
$$ \# P_{i,j} = 
\begin{cases} 
\frac{(\ell-1)} 2 \ell^{2k-1} & \textnormal{ if } i= 0 \textnormal{ and } j = 0 \\
\frac{(\ell-1)(\ell-2)} 2 \ell^{2k-2} & \textnormal{ if } i= 0 \textnormal{ and } j = 2 \\
(\ell-1)^2 \ell^{2k-2t-1} & \textnormal{ if } i= 2t-1 < k \textnormal{ and } j = 0 \\
\frac{(\ell-1)^2} 2 \ell^{2k-2t-2} & \textnormal{ if } i= 2t < k \\
(\ell-1)\ell^{k-1} & \textnormal{ if } i=k \textnormal{ and } j=2 \\
0 & \textnormal{ else}.
\end{cases} $$
\end{lemma}
\begin{proof}
Let $\ell$ be an odd prime and $k$ a positive integer. Let $f = X^2 + bX +c $ be a quadratic monic polynomial. We distinguish three cases depending on the valuation of the discriminant $D$ of $f$.
\begin{enumerate}

\item Suppose that $D$ is a unit. Then $D$ is a square modulo $\Z/\ell^k\Z$ if and only if $D$ is a quadratic residue modulo $\ell$. So it suffices  to count quadratic polynomials $f\in \F_\ell[X]$ with $f(0)$ a unit and $D$ a quadratic residue or non-residue. There are $\frac{(\ell-1)\ell}{2}$  monic quadratic irreducible polynomials with $f(0) \neq 0$ in $\F_\ell[X]$ and  $\frac{(\ell-1)(\ell-2)}{2}$ monic quadratic  polynomials with $f(0) \neq 0$ and distinct roots. The result for $\v(D)=0$  follows from Hensel's lemma.

\item Let $0<i<k$. Then  
$$f \in P_i \Leftrightarrow c \in (b/2)^2 + \ell^i(\Z/\ell^k\Z^\times).$$
In particular $b$ is a unit. For each of the  $(\ell- 1)\ell^{k-1}$  different choices of $b$, there are $(\ell-1)\ell^{k-i-1}$ choices for $c$. Hence $\#P_i = (\ell-1)^2 \ell^{2k-i-2}$.
If  $i = 2t-1$ is odd, then $P_{2t-1,2}$ is empty so  $\#P_{2t-1,0}=(\ell - 1)^2 \ell^{2k-2t-1}$. If $i = 2t$ is even then the discriminant is a square for half of the polynomials. Hence $ \#P_{2t,0} = \#P_{2t,2}= \frac{(\ell - 1)^2} 2 \ell^{2k-2t-2}$.

\item Finally suppose that  $i = k$, then $D = 0 $ if and only if $c = \frac {b^2} 4$. Hence for every unit $b$ there is  precisely one polynomial in the set $P_k$ with discriminant zero. \qedhere
\end{enumerate}
\end{proof}
For every non-empty set $P_{i,j}$ fix a polynomial $f_{i,j} \in P_{i,j}$. Denote $$M_{i,j} :=\{\sigma \in GL_2(\Z/\ell^k \Z) \mid \textnormal{char. poly.}\ \sigma = f_{i,j} \}.$$
If no such polynomial $f_{i,j}$ exists, $M_{i,j}$ is defined as the empty set.
\begin{lemma}\label{lemma:Mlk}
Let $\ell$  be an odd prime and $k$ a positive integer. Then
$$\#M_{i,j} = 
\begin{cases}
(\ell- 1)\ell^{2k-1} & \textnormal{if } i = 0 \textnormal{ and } j = 0 \\
(\ell+1) \ell^{2k-1} & \textnormal{if } i = 0 \textnormal{ and } j = 2 \\
(\ell^{t+1} + \ell^t -\ell - 1)\ell^{2k-t-1} & \textnormal{if } i = 2t-1 \textnormal{ and } j = 0\\
(\ell^{t+1} + \ell^t -2 )\ell^{2k-t-1} & \textnormal{if } i = 2t \textnormal{ and } j = 0 \\
(\ell+1 )\ell^{2k-1} & \textnormal{if } i = 2t \textnormal{ and } j = 2 \\
(\ell^{m+1} + \ell^m -1) \ell^{3m-1} & \textnormal{if } i = k = 2m  \textnormal{ and } j=2\\
(\ell^{m+1}+\ell^m-1)\ell^{3m+1} & \textnormal{if } i = k = 2m+1 \textnormal{ and } j=2 \\
0 & \textnormal{else.}
\end{cases}$$
In particular $\#M_{i,j}$ does not depend on the  polynomial $f_{i,j}$.
\end{lemma}

\begin{proof}
	If $j=2$ and $i$ is odd, the set $M_{i,j}$ is empty by definition and so is the set $M_{k,0}$. We prove the remaining cases.
	
	Let $\ell$ be an odd prime and $k$ a positive integer.  Let $f = X^2 -TX +S  \in \Z/\ell^k \Z[X]$ with  $S$ invertible. Let $D=T^2 -4S$ be the discriminant of $f$. A matrix $\begin{pmatrix} a & b \\ c & d \end{pmatrix}$ has characteristic polynomial $f$ if and only if
	\begin{align*}
	\begin{cases}d= T-a \\ a^2 - aT + S+bc = 0. \end{cases}
	\end{align*}
	So for every pair $(b,c)$ and every zero of $ X^2 - TX  +S+ bc$ there exists precisely one matrix with characteristic polynomial $f$. By Lemma \ref{lemma:DiscModlk} the number of zeros of $f+bc$ depends only on  the discriminant of $f+bc$. The discriminant of $f+bc$ is $D -4bc$. We distinguish three cases depending on the valuation of $bc$. Define
	\begin{align*}
	M_{i,j}^< &:=  \{ \sigma \in M_{i,j} \mid \v(bc) <i\},\\
	M_{i,j}^= &: = \{ \sigma \in M_{i,j} \mid \min {\v(bc)} k  =i\}, \\
	M_{i,j}^> &: =\{ \sigma \in M_{i,j} \mid \min {\v(bc)} k  >i\}.
	\end{align*}
	Then for every pair $i$ and $j$ 
	$$\# M_{i,j} = \# M_{i,j}^< + \# M_{i,j}^= +\# M_{i,j}^>.$$ 
	First we compute the cardinality of each of the sets $M_{i,j}^<$, $M_{i,j}^=$ and $M_{i,j}^>$ for all pairs $i$ and $j$. 
	\begin{enumerate}
		\item Suppose that $0 \le \v(bc) <\min {\v(D)} k =i \le k$. Then $\v(D -4bc) = \v(bc)$. We need only consider tuples $(b,c)$ such that the valuation of $bc$ is even. Indeed by Lemma \ref{lemma:DiscModlk}   monic quadratic polynomials with odd valuated discriminant have no zeros. 
		For each integer $s$ with $0\leq 2s < i $ there exist $(2s+1)(\ell-1)^2 \ell^{2k-2s-2}$ pairs $(b,c)$ such that $\v(bc) = 2s$. Moreover for exactly half of these pairs $D-4bc$ is a square. In this case there exist $2\ell^s$ solution for the polynomial $f+bc$ by Lemma \ref{lemma:DiscModlk}.
		
		So for each $0\le s<\ceil{i/2}-1$ we obtain $(2s+1)(\ell-1)^2\ell^{2k-s-2}$ different matrices with characteristic polynomial $f$. Summing over all $s$ yields 
		\begin{align*}
			&\sum_{s=0}^{\lceil i/2 \rceil -1}(2s+1)(\ell-1)^2\ell^{2k-s-2} \\
			&= 2(\ell-1)\sum_{s=0}^{\lceil i/2 \rceil -1} s(\ell-1)\ell^{2k-s-2} + (\ell-1)\sum_{s=0}^{\lceil i/2 \rceil -1} (\ell-1)\ell^{2k-s-2} \\
			& = 2(\ell-1)\Big( \ell^{2k-2} + \ell^{2k-3} + \cdots + \ell^{2k-\lceil i/2\rceil} - (\lceil i /2\rceil-1)\ell^{2k-\lceil i/2\rceil-1} \Big) \\
			& \phantom{=}+ (\ell-1)\Big(\ell^{2k-1}  - \ell^{2k - \lceil i/2 \rceil -1}\Big) \\
			& = (\ell-1)\ell^{2k- \lceil i/2 \rceil -1} \Big( 2\frac{(\ell^{ \lceil i/2\rceil} - 1 )}{\ell-1} -2\lceil i/2\rceil  +\ell^{\lceil i /2 \rceil}  - 1\Big) \\
			& = \ell^{2k-\lceil i/2 \rceil-1} \Big( 2\ell^{\lceil i/2\rceil}  -2 - 2\ceil{i/2}(\ell-1) + \ell^{\lceil i/2\rceil}(\ell-1)-(\ell-1)\Big)\\
			& = \ell^{2k-\lceil i/2\rceil - 1} \Big( \ell^{\ceil{i/2}} (\ell+1) - 2\ceil{i/2}(\ell-1) - \ell - 1  \Big).
		\end{align*}
		So we obtain 
		$$\# M_{i,j}^< = \begin{cases} 0 & \textnormal{ if } i = 0 \\ \ell^{2k-\lceil i/2\rceil - 1} \Big( \ell^{\ceil{i/2}} (\ell+1) - 2\ceil{i/2}(\ell-1) - \ell - 1  \Big)& \textnormal{else}. \end{cases}$$
		
		\item Suppose that $0 \le  \min {\v (bc)}k  = i = \min {\v(D)} k  \le k$. Then the valuation of $D-4bc$ will be at least $i$ and may be bigger depending on $bc$. 
		
		If $i<k$ there are $(i+1)(\ell-1)^2 \ell^{2k-i - 2}$ pairs $(b,c)$ such that $\v(bc) = i$. Since every element of $D+\ell^i (\Z/\ell^k\Z^\times)$ occurs an equal amount of times  as $D-4bc$ for  all $b$ and $c$ in $\Z/\ell^k \Z$ with $\v(bc) = i$ it suffices to count the number of squares with given valuation. In particular each element of $D + \ell^i (\Z/\ell^k\Z^\times)$ will occur precisely 
		\begin{align*}
		\frac{\# \left\{ (b,c)\in (\Z/\ell^k \Z)^2 \mid \v(bc) = i\right\}}{\# -D+\ell^i (\Z/\ell^ k \Z^\times)} & = \frac{(i+1)(\ell-1)^2\ell^{2k-i-2}} {(\ell- 1)\ell^{k-i-1}} \\ &= (i+1)(\ell-1)\ell^{k-1}
		\end{align*} times as the discriminant of $f+bc$ for all pairs $(b,c)$ with $\min {\v(bc)} k  = i$.

		If $D$ is not a square, then for any $i \le 2s <  k$ there are $\frac 1 2 (\ell-1 ) \ell^{k-2s-1}$   squares in $D + \ell^i \Z/\ell\Z^\times$ with valuation $ 2s$.  Each square with valuation $2s<k$ induces $2\ell^s$ distinct zeros and discriminant equal to zero  induces $\ell^{\floor {k/2}}$ zeros. Summing over all cases yields
		\begin{align*}
		\#M_{i,0}^==&(i+1)(\ell-1)\ell^{k-1}\Big(\ell^{\floor {k/2}} + \sum_{s = \ceil {i/2}}^{\ceil{k/2}-1} \frac{(\ell- 1)} 2 \ell^{k-2s-1} \cdot 2\ell^s\Big) \\
		& = (i+1)(\ell- 1) \ell^{k-1} \big( \ell^{\floor {k/2}} + \ell^{k-\ceil {i/2}}  - \ell^{k- \ceil{k/2} } \big) \\
		& = (i+1)(\ell- 1) \ell^{2k-\ceil{i/2}-1}.
		\end{align*}

		If $D$ is a square, then there are only $\frac{1}{2}(\ell-3)\ell^{k-i-1}$ squares with valuation $i$ as all elements in $D+\ell^{i+1}\Z/\ell^k\Z$  are quadratic residues modulo $\ell^k$ with valuation $i$ and  $$\Big(D+\ell^{i+1}\Z/\ell^k\Z  \Big)\cap \Big(D + \ell^i (\Z/\ell\Z^\times)\Big)= \varnothing. $$
		The number of squares with valuation $2s>i$ is the same as in the case that $D$ is a quadratic non-residue modulo $\ell^k$. So summing over all $s$ with $i\le 2s \le k$ yields
		\begin{align*}
		\#M_{i,2}^==&(i+1)(\ell-1)\ell^{k-1} \Big(\ell^{\floor {k/2}} +(\ell- 3)\ell^{k-i/2-1}+ \sum_{s =i/2+1}^{\ceil{k/2}-1} (\ell- 1)\ell^{k-s-1} \Big) \\
		& = (i+1)(\ell- 1) \ell^{k-1} \big( \ell^{\floor {k/2}}+\ell^{k-i/2} -3\ell^{k-i/2-1} + \ell^{k-i/2-1}  - \ell^{\floor {k/2} } \big) \\
		& = (i+1)(\ell- 1)(\ell-2) \ell^{2k-i/2-2}.
		\end{align*}

		Finally suppose that  $D= 0 = bc$. Then one checks that the number of pairs $(b,c)$ such that $bc = 0$ is $ \big((k+1)\ell- k\big)\ell^{k-1}$.
		For each of these pairs the discriminant of the polynomial $f+bc$ equals the discriminant of the polynomial $f$ which is zero. In particular every pair $(b,c)$ induces $\ell^{\floor {k/2}}$ distinct matrices with characteristic polynomial $f$.

		So the number of matrices with characteristic polynomial $f$ and $ 0\le \min{\v(bc)} k  = \min {\v(D)} k =i \le k $ is 
		$$\# M_{i,j}^= = \begin{cases}
		(i+1)(\ell- 1)(\ell-2) \ell^{2k-i/2-2} & \textnormal{if } i<k \textnormal{ and } j=2\\
		\big((k+1)\ell-k\big)\ell^{k + \floor{k/2} -1} & \textnormal{if } i=k \\
		(i+1)(\ell- 1) \ell^{2k-\ceil{i/2}-1} & \textnormal{else. }
		\end{cases}$$
		\item Suppose that $0\le i=\v(D)< \min {\v(bc)} k  \le k$. If $D$ is not a square in $\Z/\ell^k\Z$, then $D-4bc$ is not a square since $D-4bc \equiv D  \mod \ell^{\v(bc)}$ and $D$ is a quadratic non-residue modulo $\ell^{\v(bc)}$. So if $D$ is not a square, no matrices exists.

		If $D$ is a square with even valuation $i$ one checks that there exists $\ell^{2k-i - 2}\big((i+2)\ell - i - 1\big)$ pairs $(b,c)$ such that $i<\min {\v (bc)} k  \le k$. For each pair there exists $2\ell^{i/2}$ zeros of the polynomial $f+bc$. Hence we find 
		$$\# M_{i,j}^> = \begin{cases}2 \ell^{2k-i/2-2}\big((i+2) \ell - i -1\big) & \textnormal{if } i<k \textnormal{ and } j = 2 \\
		0 & \textnormal{else. } \end{cases}$$ 
		\end{enumerate}
		We compute the sum $\# M_{i,j} = \# M_{i,j}^< + \# M_{i,j}^= +\# M_{i,j}^>$ for each pair $i$ and $j$.	
		\begin{enumerate}
		\item Suppose that  $\v(D) = 0$ and $D$ is a quadratic non-residue. Then there exist no matrices with characteristic polynomial $f$ unless $\v(bc) = 0$. In this case
		\begin{align*}
		 \#M _{0,0} &= (\ell-1)\ell^{2k-1}.
		 \intertext{
		\item If $\v(D) = 0 $ and $D$ is a quadratic residue. There exist $(\ell- 1)(\ell-2)\ell^{2k-2}$  matrices with $\v(bc)  = 0$ and $2\ell^{2k-2}(2\ell-1)$ with $\v(bc) >0$. Hence we obtain} 
		\# M_{0,2} &= (\ell +1)\ell^{2k-1}.
		\intertext{
		\item If $0<i<k$ is odd, say $ i = 2t -1$. Then 
		}
		\#M_{2t-1,0} =& \ell^{2k-t-1}\Big(\ell^t(\ell+1) - 2 t(\ell-1) - \ell - 1 \Big)  + 2t(\ell -1)\ell^{2k-t-1} \\
		& = (\ell^{t+1} + \ell^t - \ell -1)\ell^{2k- t - 1}.
		\intertext{
		\item If $ i$ is even say $i=2t$ and $D$ is not a square. We obtain
		}
		\# M_{2t,0} &= \ell^{2k - t - 1}\big(\ell^t (\ell+1) - 2 t (\ell-1) - \ell -1 \big) + (2t + 1)(\ell -1)\ell^{2k - t - 1} \\
		& =  (\ell^{t+1} + \ell^t -2)\ell^{2k - t- 1}.
		\intertext{
		\item If $i = 2t$ and $D$ is a square. Then summing over all cases yields
		}
		\#M_{2t,2} &= \ell^{2k - t-1}\big(\ell^t(\ell+1)-2t(\ell-1)- \ell -1 \big)   \\ &+ \ell^{2k-t-2}(2t+1)(\ell-1)(\ell-2) + 2\ell^{2k - t- 2}\big((2t+2)\ell - 2t -1\big) \\
		& =(\ell + 1)\ell^{2k-1}.
		\intertext{
		\item If $D=0$ and $k = 2m$. We obtain
		}
		\#M _{k,2} &= \ell^{4m - m - 1}\big(\ell^m(\ell+1)- 2m(\ell-1)- \ell -1 \big) \\
		&+ \ell^{2m +m - 1}\big((2m+1)\ell -2m \big) \\
		& = \ell^{3m - 1}(\ell^{m+1} + \ell^{m} -1).
		\intertext{
		\item If $D = 0 $ and $k=2m+1$. Then
		}
		 \#M _{k,2} &= \ell^{4m+2 - m - 1 - 1}\big(\ell^{m+1}(\ell+1)- 2(m+1)(\ell-1)- \ell -1 \big)\\ &  + \ell^{2m+1 +m - 1}\big((2m+2)\ell -2m-1 \big) \\
		&= \ell^{3m+1}(\ell^{m+1}+ \ell^m-1).
		\end{align*}
	\end{enumerate}
	In particular the cardinality of $M _{i,j}$ does not depend on the choice of $f_{i,j}$.
\end{proof}

\begin{proof}[Proof of Proposition \ref{prop:GLKTSplit}]
Recall that $$\A_{\ell^k,S}^t  = \{(\tau, \tau') \in GL_{2}(\Z/\ell^k \Z)^2 \mid   \textnormal{char. poly.}\  \tau = \textnormal{char. poly.}\  \tau' \}.$$
So 
\begin{align*}\#\A_{\ell^k,S}^t &= \sum_{f\in \Z/\ell^k \Z[X]} \left(\#\{ \tau \in GL_2(\Z/\ell^k \Z) \mid \textnormal{char. poly.}\ \tau = f \}\right)^2 \\
&= \sum_{i,j} \#P_{i,j} \cdot (\# M_{i,j})^2
 \end{align*}
where the sum is taken over all pairs  $(i,j)$ with $0\le i< k$ and $j = 0,2$ and the pair $(i,j)=(k,2)$. The factors $\#P_{i,j}$ and $\#M_{i,j}$ are computed in  Lemmas \ref{lemma:Plk} and \ref{lemma:Mlk}  respectively.  We will only give the proof if $k$ is odd. If $k$ is even the computation is similar. Suppose that $k=2m+1$. Then
\begin{align*}
\#\A_{\ell^k,S}^t =& \frac{(\ell-1)} 2 \ell^{2k-1} \cdot (\ell-1)^2  \ell^{4k-2} + \frac{(\ell-1)(\ell-2)} 2 \ell^{2k-2} \cdot (\ell+1)^2 \ell^{4k-2} \\
&+ \sum_{t=1}^{m} \Big(  (\ell-1)^2 \ell^{2k-2t-1} \cdot \big(\ell^t(\ell+1) -(\ell+1) \big)^2 \ell^{4k-2t-2} \\
 &+ \frac{(\ell-1)^2} 2 \ell^{2k- 2t-2} \cdot\big( \ell^t(\ell+1) - 2 \big)^2\ell^{4k-2t-2} \\
 &+ \frac{(\ell-1)^2} 2 \ell^{2k- 2t-2} \cdot(\ell +1)^2 \ell^{4k-2} \Big) \\
 & + (\ell-1)\ell^{2m} \cdot \big(\ell^m(\ell+1) - 1 \big)^2 \ell^{6m+2} \\
 =&  \frac{(\ell-1)} 2 \ell^{6k-4}  \big( (\ell-1)^2\ell + (\ell+1)^2 (\ell-2) \big)\\
 & + \sum_{t=1}^{m} \ell^{6k - 4t-4} \Bigg(  \frac{(\ell-1)^2}{2} \big( 2\ell^{2t+1}(\ell+1)^2 - 4 \ell^{t+1}(\ell+1)^2 + 2\ell(\ell+1)^2 \big) \\
 & +  \frac{(\ell-1)^2} 2\big( \ell^{2t}(\ell+1)^2 - 4 \ell^t(\ell+1) + 4\big) \\
  & +  \frac{(\ell-1)^2} 2 \ell^{2t}(\ell+1)^2 \Bigg) \\
 & + (\ell-1)\ell^{8m + 2}\big(\ell^{2m}(\ell+1)^2 - 2\ell^m(\ell+1) +1 \big) \\
 = &\frac{(\ell-1)} 2 \ell^{6k-4} (2\ell^3 - 2 \ell^2-2\ell-2) \\
 & + \frac{(\ell-1)^2} 2 \sum_{t=1}^{m} \ell^{6k - 4t-4}\Big( 2\ell^{2t} ( \ell+1)^3  - 4 \ell^t(\ell+1)(\ell^2 + \ell + 1) + 2 (\ell^2+1)(\ell+2) \Big) \\
 & + (\ell-1)\big((\ell+1)^2 \ell^{10m+2} - 2(\ell+1)\ell^{9m+2} + \ell^{8m+2} \big).
 \intertext{Splitting the summation into three  sums yields}
 \#\A_{\ell^k,S}^t  = &(\ell-1) \ell^{6k-4} (\ell^3 -  \ell^2-\ell-1) \\
 & + (\ell-1)(\ell+1)^2\sum_{t = 1}^{m} (\ell^2-1)\ell^{6k - 2t- 4} \\
 & -2(\ell-1)(\ell+1)\sum_{t = 1}^m (\ell^3-1)\ell^{6k-3t-4} \\
 & + \frac{(\ell-1)(\ell+2)} {(\ell+1)} \sum_{t = 1}^{m} (\ell^4-1)\ell^{6k - 4 t -4} \\
 & + (\ell-1)\big((\ell+1)^2 \ell^{10m+2} - 2(\ell+1)\ell^{9m+2} + \ell^{8m+2} \big).
 \intertext{Computing the telescopic sums and using that $k = 2m+1$ }
 \#\A_{\ell^k,S}^t = &(\ell-1) \ell^{12m+2} (\ell^3 -  \ell^2-\ell-1) \\
 & + (\ell-1)(\ell+1)^2(\ell^{12m+2} - \ell^{10m+2}) \\
 & -2(\ell-1)(\ell+1)(\ell^{12m+2} - \ell^{9m+2})\\
 & + \frac{(\ell-1)(\ell+2)} {(\ell+1)} (\ell^{12m+2} - \ell^{8m+2}) \\
 & + (\ell-1)\big((\ell+1)^2 \ell^{10m+2} - 2(\ell+1)\ell^{9m+2} + \ell^{8m+2} \big).
 \intertext{By sorting the powers of $\ell^m$ we obtain}
 \#\A_{\ell^k,S}^t =  & (\ell-1)\ell^{12m+2} \Big((\ell^3-\ell^2-\ell-1) + (\ell+1)^2 - 2(\ell+1) + \frac{(\ell+2)}{(\ell+1)} \Big) \\
  & + (\ell-1)(\ell+1)^2\ell^{10m+2}(-1+1) \\
  &- 2 (\ell-1)(\ell+1)\ell^{9m+2}(-1+1) \\
  & +(\ell-1)\ell^{8m+2}\Big(-\frac{(\ell+2)}{(\ell+1)}+1\Big) \\
  = & (\ell-1)\ell^{12m+2}\Big(\ell^3 - \ell - 2 +\frac{(\ell+2)}{(\ell+1)} \Big) - \frac{(\ell-1)}{(\ell+1)}\ell^{8m+2} \\
  =& \frac{(\ell-1)}{(\ell+1)}\ell^{12m+2} (\ell^4 + \ell^3 -\ell^2-2\ell - \ell^{-4m}).
  \intertext{Using that $k = 2m+1$ yields}
  \#\A_{\ell^k,S}^t = & \frac{(\ell-1)}{(\ell+1)}\ell^{6k-4} (\ell^4 + \ell^3 -\ell^2-2\ell - \ell^{-2k+2}). \qedhere
\end{align*}
\end{proof}

\bibliography{LTGBib.bbl}
\bibliographystyle{plain}
\Addresses
\end{document}